\documentclass[a4paper,reqno, 11pt]{amsart}

\usepackage{amsmath}
\usepackage{paralist}
\usepackage{graphics}
\usepackage{epsfig}
\usepackage{amsfonts}
\usepackage{amssymb}
\usepackage{color}
\usepackage{epstopdf}

\newtheorem{theorem}{Theorem}[section]
\newtheorem{lemma}[theorem]{Lemma}

\def\R{{\Bbb R}}

\def\ifl{\iffalse }

\def\bc{\begin{center}}       \def\ec{\end{center}}
\def\ba{\begin{array}}        \def\ea{\end{array}}
\def\be{\begin{equation}}     \def\ee{\end{equation}}
\def\bea{\begin{eqnarray}}    \def\eea{\end{eqnarray}}
\def\beaa{\begin{eqnarray*}}  \def\eeaa{\end{eqnarray*}}

\numberwithin{equation}{section}

\newtheorem{corollary}[theorem]{Corollary}
\newtheorem{remark}[theorem]{Remark}
\numberwithin{equation}{section}

\newcommand{\D}{\displaystyle}

\begin{document}

\title[Chemotaxis system, boundedness, and global existence ]
{A  class of chemotaxis systems with   growth  source  and nonlinear secretion}

\author{Zhi-an Wang}
\address{Department of Applied Mathematics, Hong Polytechnic University, Hong Kong}

\author{Tian Xiang}
\address{Institute for Mathematical Sciences, Renmin University of China, Beijing,  100872, China}
\email{mawza@polyu.edu.hk}
\email{txiang@ruc.edu.cn}

\subjclass[2000]{Primary:  35K57, 	35K51,  92C17; Secondary: 37K50, 35A01.}


\keywords{Chemotaxis systems, growth source, nonlinear secretion, characterization,  global existence, pattern formation, stability. }

\begin{abstract}
In this paper, we are concerned with a class of  parabolic-elliptic chemotaxis systems encompassing the prototype
$$\left\{ \begin{array}{lll}
&u_t = \nabla\cdot(\nabla u-\chi u\nabla v)+f(u),  & x\in \Omega, t>0, \\[0.2cm]
&0= \Delta  v -v+u^\kappa,  &  x\in \Omega, t>0   \end{array}\right.
$$
with nonnegative initial  condition for $u$  and homogeneous Neumann boundary conditions in a smooth bounded domain $\Omega\subset \mathbb{R}^n(n\geq 2)$, where  $\chi>0$,   $\kappa>0$ and $f$ is a smooth growth source satisfying $f(0)\geq 0$ and
$$
f(s)\leq a-bs^\theta, \quad s\geq 0, \text{ with some } a\geq 0, b>0, \theta>1.
$$
Firstly, it is shown,  either
$$
\kappa<\frac{2}{n}\quad  \&  \quad f\equiv 0,
$$
or
$$\theta>\kappa+1,
$$
or
$$
\theta-\kappa=1, \ \ b\geq \frac{(\kappa n-2)}{\kappa n}\chi, \eqno(*)
$$
 that the corresponding initial-value problem admits a unique classical solution that is uniformly bounded in space and time. Our proof is elementary and  semigroup-free. Whilst,  with the particular choices $\theta=2$ and $\kappa=1$, Tello and Winkler \cite{TW07} use sophisticated estimates via the Neumann heat semigroup to obtain the global boundedness under the strict inequality in ($\ast$).  Thereby,  we  improve their results to the ``borderline" case   $b=(\kappa n-2)/(\kappa n)\chi$  in this regard. Next, for an unbounded range of  $\chi$, the system is shown to exhibit pattern formations, and,   the emerging  patterns are shown to converge weakly  in $ L^\theta(\Omega)$   to some constants as $\chi\rightarrow \infty$.  While,  for small $\chi$ or large damping $b$, precisely $b>2\chi$ if $f(u)=u(a-bu^\kappa)$ for some $a, b>0$, we show that the system does not admit pattern formation and  the large time behavior of solutions is comparable to its associated  ODE+algebraic system.

\end{abstract}

\maketitle

\section{Introduction}


Following the first chemotaxis model proposed by Keller and Segel in \cite{Ke} to describe the aggregation phase of cellular slime mold, the mathematical modeling and analysis of chemotaxis have been rapidly  developed in various deep ways (see review articles \cite{BBTW15, Hi, Ho1, Wang-review}). Due to its important applications in biological and medical sciences, chemotaxis research has become one of the most hottest topics in applied mathematics nowadays and tremendous theoretical progresses have been made in the past few decades.  This paper is devoted to making further development  for the following  quasilinear parabolic-elliptic chemotaxis systems with nonlinear production of signal and  growth source, reading as
\begin{equation}\label{para-elli}
\left\{ \begin{array}{llll}
&u_t = \nabla \cdot (\nabla u-\chi u\nabla v)+f(u),  & x\in \Omega, t>0, \\[0.2cm]
&\tau v_t= \Delta v  - v+g(u)  ,  & x\in \Omega, t>0, \\[0.2cm]
&\frac{\partial u}{\partial \nu}=\frac{\partial v}{\partial \nu}=0, & x\in \partial \Omega, t>0,\\[0.2cm]
&u(x,0)=u_0(x), \  \tau v(x,0)=\tau v_0(x), & x\in \Omega   \end{array}\right.  \end{equation}
with $\tau \in \{0,1\}$, where $\Omega\subset \mathbb{R}^n (n\geq 1$) is a bounded domain  with  smooth boundary $\partial \Omega$ and $\frac{\partial }{\partial \nu}$ denoted the derivative with respect to the outward normal vector $\nu$ of $\partial\Omega$.  $u(x,t)$ and $v(x,t)$ denote the cell density and chemical concentration, respectively.   $\chi(>0)$ is referred to as the chemotactic sensitivity coefficient measuring the strength of chemotaxis. The  kinetic term $f$ describes cell proliferation and  death (simply referred to as growth) and $g(u)$ accounts for the chemical secretion by cells. If $\tau=1$, the model (\ref{para-elli}) is called full parabolic-parabolic chemotaxis system.  If $\tau=0$, (\ref{para-elli}) is referred to as simplified parabolic-elliptic chemotaxis system which is physically  relevant when the chemicals diffuse much faster than cells do. This simplified system was first introduced for the case $f(u)=0$  and $g(u)= u$ (minimal model) in \cite{Jager} and thereafter was studied by other authors in various contexts (e.g. see \cite{Na95, Herrero-Vela97b, TW07}).

It has been well-known that when $f(u)=0$ and $g(u)= u$, the minimal model (\ref{para-elli}) possesses blow-up solutions in finite/infinite time (see \cite{Ho1,Win100,  Win13}) in two or higher dimensions. This limits the value of the model to explain the aggregation phenomena observed in experiment. Hence,  the foremost question for the chemotaxis-growth system (\ref{para-elli}) is whether or not the appearance of growth source $f(u)$ can enforce the boundedness of solutions so that blow-up is inhibited. Toward this end, many efforts have been made first for the linear chemical production and the logistic source:
\begin{equation}\label{form}
f(u)=ru-\mu u^2, \ \ \ \ \ g(u)=u
\end{equation}
First, Osaki {\it et al} \cite{OTYM02} proved that in two dimension ($n=2$) the model (\ref{para-elli}) with $\tau=1$ and (\ref{form}) has a classical uniform-in-time bounded solution for any $r\in \R, \mu>0$. In higher dimensions ($n\geq 3$), Winkler \cite{Win10} shows, under the logistic source
\begin{equation}\label{log-form}
f(u)\leq a-bu^2, \ \ f(0)\geq 0, \ \ \ a\geq 0, b>0, u\geq 0,
\end{equation}
there exists a large positive number  $b_0$ such that if $b>b_0$, then the chemotaxis-growth  system (\ref{para-elli}) with $\tau=1$ and $g(u)=u$  admits a classical uniform-in-time bounded solution in a bounded  convex domain $\Omega \subset \R^n$.  The  existence of global weak solutions to \eqref{para-elli} with $\tau=1$ and \eqref{form}  is newly known   for  $\mu>0$  in  convex domains \cite{La15-JDE}.  Recently,  it was further proved in \cite{Win15-jde},  if $\mu$ is sufficiently large, the solution $(u,v)$ of \eqref{para-elli}+\eqref{form} stabilizes to the constant steady state $(\frac{r}{\mu}, \frac{r}{\mu})$ globally as time tends to infinity. However,  the explicit form of $b_0$ for the parabolic-parabolic system (\ref{para-elli}) (i.e. $\tau=1$) is largely open today. A rough explicit lower bound for a 3-D chemotaxis-fluid system with logistic source was obtained in \cite{TW15-ZAMP}, when applied to  the chemotaxis system \eqref{para-elli}+\eqref{form} with $\chi=\tau=1$, their result states that $\mu\geq 23$ is enough to prevent blow-ups.  But finding the explicit form of $b_0$, in particular,  the possible smallest value of $b_0$, is a very interesting question since it addresses how strong the dampening source is needed to inhibit the blowup; an effort toward this is under exploration \cite{Xiang15-pre}.  Turning to parabolic-elliptic systems, some progress for (\ref{para-elli}) (i.e.,  $\tau=0$) has been made by Tello and Winkler \cite{TW07} wherein they showed, if $b >b_0=\frac{n-2}{n}\chi$, then the system (\ref{para-elli}) with \eqref{log-form} and $g(u)=u$ admits globally bounded classical solutions via sophisticated estimates via the Neumann heat semigroup. Whilst,  it has not been known whether or not their results hold true for the borderline case $b =b_0$. The starting  point of this paper is then to extend the results of \cite{TW07} to the borderline case $b=b_0$ by using elementary mathematical tools instead of complicated semigroup approach.  On the other hand,   the chemical production function $g(u)$ could  not necessarily be linear in $u$ (that is $\kappa$ may not be $1$). For example, $g(u)=u^2$ was used to model the aggregation patterns formed by bacterial chemotaxis (see \cite[Chapter 5]{Mu02}). Indeed, in a series of works by Nakaguchi {\it et al} \cite{EN08, NO11, NO13}, the following functions $f(u)$ and $g(u)$ are chosen:
\begin{equation}\label{form1}
f(u)=u-\mu u^\theta, \ \ g(u)=u(u+1)^{\kappa-1}, \ \mu>0, \theta>1, \kappa>0.
\end{equation}
It was shown in \cite{EN08, NO11, NO13} that  if either
$$\theta>2 \kappa+1 \ \mathrm{for} \ n\geq 3$$
or
$$\theta \geq\max\{2, 2\kappa\} \ \mathrm{for}\ n=2,$$
then the system (\ref{para-elli}) with $\tau=1$ and (\ref{form1}) has at least one global classical solution. The above conditions require rapid dampening  source term (i.e.,  the power parameter $\theta$ is required to be large).  Accordingly,  it would be interesting to ask whether this exponent can be reduced to guarantee the existence of globally bounded solutions. Finally,  we observe that enormous  variants of (\ref{para-elli})+\eqref{form} have been considered  (e.g. see \cite{ BH13,  Cao14,  TW12,  TW14-JDE, Win08,  Xiang15, ZL15-ZAMP, Zh15}), and  that explosion of solutions  is still possible  in chemotaxis systems despite logistic dampening \cite{Win11}.

 Beyond  the boundedness issue motivated above, we also wish to study qualitative properties for the chemotaxis model  (\ref{para-elli}). Thus, in this work,  we propose to consider  the chemotaxis-growth model (\ref{para-elli}) with $\tau=0$ (the case $\tau=1$ is largely untouched), and the following more general conditions of $f(u)$ and $g(u)$ covering the commonly used choices  (\ref{form}) and (\ref{form1}): $f$ is smooth with $f(0)\geq 0$ and there are $a\geq 0, b>0$ and $\theta>1$ such that
\begin{equation}\label{f-con}
\ \ \ \ \ \ \ \ \ \ \ \ \ \ \ \ \ \ \ \ \ \ f(u)\leq a-bu^\theta \text{ for all } u\geq 0
\end{equation}
and, $g$ is smooth and there are $\beta>0$ and $\kappa>0$ such that
\begin{equation}\label{g-con}
\ \ \ \ \ \ \ \ \ \ g(u)\leq \beta u^\kappa \text{ for all }  u\geq 0.
\end{equation}
Note that  the assumption (\ref{f-con}) on $f$ recovers a wide class of biological meaningful source functions, such as the logistic type $f(u)= r u-\mu u^2 (r\geq 0, \mu >0)$ for $\theta=2$ and Allee effect type $f(u)=u(1-u)(u-1/2)$ for  $\theta=3$.
\medskip

The main results of this work are outlined below:

\begin{itemize}
\item \textbf{Global classical solutions}. In Section 4, with  the aid of the $L^{\kappa n/2+\epsilon}$-boundedness  criterion  in Theorem \ref{thm-para-para-global-ext},  we establish the boundedness and  global existence of classical solutions to the system (\ref{para-elli}) with $\tau=0$ and (\ref{f-con})-(\ref{g-con}) if either
$$\theta>\kappa+1,
$$
or
$$
\theta=\kappa+1, \ \ b\geq b_0=\frac{(\kappa n-2)}{\kappa n}\chi.
$$
These results are stated in Theorems  \ref{thm-para-elli-min} and \ref{thm-para-elli-min2}. As a by-product, we extend the global existence results for the linear secretion   $g(u)=u$ of Tell and Winkler \cite{TW07} to the borderline case $b=b_0$ via a more elementary and semigroup-free method, and hence also generalizes the boundedness results of \cite{CZ13, WMZ14, Zh15}.  Clearly, our results also largely improve the boundedness  results of \cite{EN08, NO11, NO13} by reducing the dampening strength parameter $\theta$.

\item \textbf{Non-constant steady states}. In  Section 5,  we first study the regularity and then show the existence of non-constant steady states of (\ref{para-elli}) by the degree theory, which not only covers the results of Tello and Winkler \cite{TW07} with  logistic source and linear production (i.e.  $\kappa=1$), but also provide  clearer  and more verifiable conditions for the existence of pattern formations (Theorem \ref{bif-thm}).  More importantly,  we investigate the asymptotic behavior of stationary solutions as $\chi \to \infty$ in certain parameter regime, which demonstrates that  the emerging patterns will converge weakly in $L^\theta(\Omega)$ to some constants when chemotatic effect becomes highly strong , cf. Theorem \ref{asy-profile}. This also offers a clarification for the gaps  made in Kuto {\it et al} \cite{KOST12}  for the special cases $f(u)=au-bu^2$ and $g(u)=\beta u$.

\item  \textbf{Large time behavior of solutions}. In Section 6, to extend Tello and Winkler's argument \cite{TW07} to general cases,  we first specify the choices  $f(u)=u(a-bu^\kappa)$ and $g(u)=u^\kappa$, under the explicit condition  $b>2\chi$, we show that the constant steady state $((\frac{a}{b})^{\frac{1}{\kappa}}, \frac{a}{b})$  is globally  asymptotically stable with exponential decay rate, cf.  Theorem \ref{large time}. Then,  for a general cell kinetic form $f(u)$, we are also able to show that the constant steady state is locally exponentially asymptotically stable under certain conditions, cf. Theorem  \ref{large time2}.
\end{itemize}

\section{Preliminaries}
To start with, let us collect  a well-known calculus inequality  and  state the local well-posedness  of the chemotaxis-growth system (\ref{para-elli}).

\begin{lemma}[Gagliardo-Nirenberg interpolation inequality \cite{Nirenberg66, Fried}]\label{GN-general}
Let $\Omega$ be a bounded domain in $\mathbb{R}^n$ with a smooth boundary and let
 $1 \leq  q, r\leq \infty$.
 \begin{itemize}
 \item [(i)]For any number  $\delta\in (0,1)$, set
 $$
 \frac{1}{p}=\delta\bigg(\frac{1}{r}-\frac{1}{n}\bigg)+(1-\delta)\frac{1}{q}.
 $$
 Then
 \be\label{GN-gen}
 \| w\|_{L^p(\Omega)}\leq C\|w\|_{W^{1,r}(\Omega)}^\delta\|w\|_{L^q(\Omega)}^{1-\delta}, \quad \forall \  w\in W^{1,r}(\Omega)\cap L^q(\Omega).
 \ee
 \item[(ii)] For any number  $\delta\in [1/2,1)$, set
 $$
 \frac{1}{p}=\frac{1}{n}+\delta\bigg(\frac{1}{r}-\frac{2}{n}\bigg)+(1-\delta)\frac{1}{q}.
 $$
Then
 \be\label{GN-geng}
 \|\nabla   w\|_{L^p(\Omega)}\leq C\|w\|_{W^{2,r}(\Omega)}^\delta\|w\|_{L^q(\Omega)}^{1-\delta}, \quad \forall \ w\in W^{2,r}(\Omega)\cap L^q(\Omega).
 \ee
  The constant depends only on $\Omega, q,r, \delta, n$.
 \end{itemize}
\end{lemma}
We will mostly use the following case: for $p,r\geq 1$ satisfying $p(n-r)<nr$ and all $q\in (0, p)$, then (i) of Lemma \ref{GN-general} holds for
$$
\delta=\displaystyle \frac{\frac{1}{q}-\frac{1}{p}}{\frac{1}{n}+\frac{1}{q}-\frac{1}{r}}\in(0, 1).
$$
The  local-in-time existence of classical solutions to  the chemotaxis-growth  system (\ref{para-elli}) is quite standard; see similar discussions in \cite{TW07, WD10, TW12,  WWW12,  CZ13,  WMZ14,  Zh15}.

\begin{lemma}\label{local-in-time}
Let  $\Omega\subset \mathbb{R}^n$ be a bounded  and smooth domain, the  nonnegative initial data $(u_0, v_0)\in (C(\overline{\Omega}), W^{1,q}(\Omega))$ for some $q>n$ and the growth source $f\in W_{\mbox{loc}}^{1,\infty}([0,\infty))$ with $f(0)\geq 0$. Then there is  a maximal existence time $T_m\in(0, \infty]$ and  a uniquer pair of nonnegative functions $(u,v)\in  C(\overline{\Omega}\times [0, T_m))\times  C^{2,1}(\overline{\Omega}\times (0, T_m))$ solving \eqref{para-elli} classically in $\Omega\times [0,T_m)$.  In particular,  if $T_m<\infty$, then
\be\label{T_m-cri} \|u(\cdot, t)\|_{L^\infty(\Omega)} +\tau \|v(\cdot,t)\|_{W^{1,q}(\Omega)} \rightarrow \infty \quad \mbox{ as } t\rightarrow T_m-.
\ee
Moreover, the $L^1$-norm of $u$ is uniformly bounded, i.e., there exists a constant $M_0$ such that $\|u(t)\|_{L^1}\leq M_0$.
\end{lemma}
\begin{proof} As noted above, the assertions concerning the local-in-time existence of  classical solutions to the IBVP (\ref{para-elli})  and the criterion (\ref{T_m-cri}) are well-studied. Since $f(0)\geq 0$, the maximum principle asserts that both  $u$ and $v$ are nonnegative, as shown in \cite{Xiang15}. Integrating the $u$-equation in (\ref{para-elli}) and using \eqref{f-con}, one can easily deduce that
$$
\frac{d}{dt}\int_\Omega u =\int_\Omega f(u)\leq \int_\Omega a-bu^\theta \leq - \int_\Omega u+c|\Omega|,
$$
where $c=\max\{a-bu^\theta+u:u\geq 0\}<\infty$ thanks to the fact that $\theta>1$. Solving this standard Gronwall's inequality shows that $L^1$-norm of $u$ is uniformly bounded.
\end{proof}

\section{A  boundedness criterion for the chemotaxis system }

For the chemotaxis model without growth, we know  that the total cell mass is conservative. This is no longer true for the chemotaxis  model with growth. However,  the total mass of cells,  $\|u(t)\|_{L^1}$,  is still uniformly bounded (cf. Lemma \ref{local-in-time}). But it is well-known that the uniform boundedness of $\|u(t)\|_{L^1}$ is not enough  to prevent the blowup of solutions in finite time (e.g. see \cite{Na95, Win11, Win13}). Inspiring by the works \cite{BBTW15, Xiang15}: we are asking whether or not the boundedness of $\|u(t)\|_{L^p}$ for some finite $p$ can ensure the boundedness of solutions. If yes, how large should $p$ be?  In other words,  if $\|u(t)\|_{L^\infty}$ blows up at $t=T_m$,  will  $\|u(t)\|_{L^p}$ also blow up  for some $p$?  This question  is  meaningful due to the   fact:
    $$ L^\infty(\Omega)\varsubsetneqq\bigcap_{p=1}^\infty L^p(\Omega). $$
The answer to these questions  will surely shed light on the mechanism of finite-time blowup. Our next result reveals an interesting characterization on the prevention of blow-up; more precisely, it asserts that the uniform boundedness of $L^p$-norm of $u(t)$  for some $p>\kappa n/2$  can rule out the blow-up of solutions.

\begin{theorem}[Criterion for boundedness]\label{thm-para-para-global-ext}
Assume that the hypotheses  \eqref{f-con} and \eqref{g-con}  hold, and $\Omega\subset \mathbb{R}^n (n\geq2$) is a bounded domain with smooth boundary. Let $(u_0,v_0)$ be as in Lemma  \ref{local-in-time} and  $(u,v)$ be the unique maximal solution of (\ref{para-elli}) defined on $[0,T_m)$. If  there exist a sufficiently small number $\epsilon>0$ and a constant $M=M(\epsilon, \kappa,  n, \Omega)>0$ such that
$$
\|u(\cdot, t)\|_{L^{\frac{\kappa}{2}n+\epsilon}(\Omega)}\leq M, \quad \forall t\in (0,T_m),
$$
then $(u(\cdot, t),v(\cdot, t))$ is uniformly bounded in $L^\infty(\Omega)\times W^{1,\infty}(\Omega)$ for all $t\in (0,T_m)$, and so $T_m=\infty$; that is, the solution $(u,v)$ exists globally with uniform-in-time bound.
\end{theorem}

\begin{proof} Multiplying the $u$-equation in (\ref{para-elli}) by $u^{p-1} (p\geq 2)$ and integrating over $\Omega$ by parts, using  Young's inequality with $\epsilon$ and  the growth condition \eqref{f-con},  we conclude  that
\begin{align*}
&\frac{1}{p}\frac{d}{dt} \int_\Omega u^p\\
&=-\int_\Omega \nabla u \nabla (u^{p-1})+\chi \int_\Omega  u\nabla  (u^{p-1})\nabla v+\int_\Omega f(u)u^{p-1} \\[0.25cm]
&\leq  -\frac{4(p-1)}{p^2}\int_\Omega |\nabla  (u^{\frac{p}{2}})|^2+\frac{2(p-1)\chi}{p} \int_\Omega u^{\frac{p}{2}}|\nabla  (u^{\frac{p}{2}})||\nabla v|+\int_\Omega f(u)u^{p-1}\\[0.25cm]
&\overset{\text{Young}} \leq -\frac{2(p-1)}{p^2}\int_\Omega |\nabla  (u^{\frac{p}{2}})|^2+\frac{ (p-1)\chi^2}{2} \int_\Omega u^p|\nabla v|^2+\int_\Omega u^{p-1}(a-bu^\theta), \end{align*}
which, upon the substitution $w=u^{\frac{p}{2}}$, reads as
\be\label{lp/2-to-lp}\begin{split}
\frac{1}{p}\frac{d}{dt} \int_\Omega w^2&\leq -\frac{2(p-1)}{p^2}\int_\Omega |\nabla w|^2 \\[0.25cm]
& +\frac{ (p-1)\chi^2}{2} \int_\Omega w^2|\nabla v|^2+\int_\Omega(aw^{\frac{2(p-1)}{p}}-bw^{\frac{2(p+\vartheta)}{p}}),
\end{split}\ee
where and hereafter, we will denote $\vartheta=\theta-1>0$.

Below we shall apply the Gagliardo-Nirenberg interpolation inequality to  control  the second integral on  the right-hand side of (\ref{lp/2-to-lp}).

Since $\epsilon>0$ is small and $n\geq 2$, it is easy to see
\be\label{r-con}
\frac{2\kappa n}{n+2}<r=\frac{\kappa}{2}n+\epsilon<\kappa n.
\ee
By the assumption that $\|u(t)\|_{L^r}$  is bounded, then $\|g(u)\|_{L^{r/\kappa}}$ is bounded due to the fact that $g(u)\leq \beta u^\kappa$.  Since this work mainly focuses on the parabolic-elliptic case and we wish to avoid the complicated semigroup theory, we will give the proof for the case $\tau=0$ (The case $\tau=1$ is noted in Remark \ref{para-para-try}). Then the use of  the point-wise elliptic $W^{2,q}$-estimate  to the $v$-equation in \eqref{para-elli} gives $\|v(t)\|_{W^{2,r/\kappa}}$ is bounded. This in turn entails by Sobolev embedding that $\|v(t)\|_{W^{1,q^\prime}}$ is bounded with
\be\label{q-exp-gl}
q^\prime= \frac{nr}{\kappa  n-r}>2
 \ee
by \eqref{r-con}. Then we obtain  from H\"{o}lder inequality  that
 \be\label{HGN-gl0}
\|w^2|\nabla v|^2\|_{L^1}\leq \|w^2\|_{L^q}\||\nabla v|^2\|_{L^{q^\prime/2}}= \|w\|_{L^{2q}}^2\|\nabla v\|_{L^{q^\prime}}^2\leq C\|w\|_{L^{2q}}^2
\ee
 with
 \be\label{q-exp-gl}
 q=\frac{\frac{q^\prime}{2}}{\frac{q^\prime}{2}-1}=\frac{nr}{(n+2)r-2\kappa n}>1.
 \ee
An application of the Gagliardo-Nirenberg  inequality (\ref{GN-gen}) to \eqref{HGN-gl0} gives
\be\label{HGN-gl}
\|w^2|\nabla v|^2\|_{L^1}\leq C\|w\|_{L^{2q}}^2
\leq C\|w\|_{W^{1,2}}^{2\delta}\|w\|_{L^\frac{2(p+\vartheta)}{p}}^{2(1-\delta)}
\ee
with
\be\label{del-qprime}
\delta=\frac{\frac{np}{2(p+\vartheta)}-\frac{n}{2q}}{1-\frac{n}{2}+\frac{np}{2(p+\vartheta)}}
=\frac{n[p(q-1)-\vartheta]}{q[2p-(n-2)\vartheta]}=\frac{2(\kappa n-r)p-[(n+2)r-2\kappa n]\vartheta}{[2p-(n-2)\vartheta]r}.
\ee
Since $r>\kappa n/2$, a simple calculation from  \eqref{del-qprime} shows
\be\label{p-exp}
p>\max\Bigr\{\frac{(n-2)\vartheta}{2}, \frac{[(n+2)r-2\kappa n]\vartheta}{2(\kappa n-r)}\Bigr\}\Longrightarrow \delta\in (0,1).
\ee
Hence, for  any $p\geq 2$ fulfilling
\be\label{p-exp2}
p>\max\Bigr\{\frac{(n-2)\vartheta}{2}, \frac{[(n+2)r-2\kappa n]\vartheta}{2(\kappa n-r)}, \kappa n\Bigr\},
\ee
 the estimate \eqref{HGN-gl} holds. Then applying Young's inequality, we conclude from \eqref{HGN-gl}  that
\be\label{Youngt-gl}
\ba{ll}
\|w^2|\nabla v|^2\|_{L^1}\leq &C\|w\|_{W^{1,2}}^{2\delta}\|w\|_{L^\frac{2(p+\vartheta)}{p}}^{2(1-\delta)}\\[0.25cm]
&\leq \epsilon_1\|w\|_{L^\frac{2(p+\vartheta)}{p}}^\frac{2(p+\vartheta)}{p}+C_{\epsilon_1}
\|w\|_{W^{1,2}}^{\frac{2\delta(p+\vartheta)}{\delta p+\vartheta}}\\[0.25cm]
&\leq \epsilon_1\|w\|_{L^\frac{2(p+\vartheta)}{p}}^\frac{2(p+\vartheta)}{p}+\epsilon_2
\|w\|_{W^{1,2}}^2+C(\epsilon_1,\epsilon_2)
\ea
\ee
for any $\epsilon_1,\epsilon_2>0$ and some constant $C$ depending on $\epsilon_1,\epsilon_2$. By Young's  inequality with epsilon, one has
\be\label{Youngt-gln}
\ba{ll}
\|w\|_{W^{1,2}}^2&=\|w\|_{L^2}^2+\|\nabla w\|_{L^2}^2\\[3mm]
&\leq \|w\|_{L^\frac{2(p+\vartheta)}{p}}^\frac{2(p+\vartheta)}{p}++\|\nabla w\|_{L^2}^2+C(|\Omega|).
\ea
\ee
Then substituting (\ref{Youngt-gln}) into (\ref{Youngt-gl}), we have
\be\label{Youngt-glnn}
\ba{ll}
\|w^2|\nabla v|^2\|_{L^1}\leq (\epsilon_1+\epsilon_2)\|w\|_{L^\frac{2(p+\vartheta)}{p}}^\frac{2(p+\vartheta)}{p}+\epsilon_2
\|\nabla w\|_{L^2}^2+C(\epsilon_1,\epsilon_2, |\Omega|).
\ea
\ee
Thus, for  $p$ satisfying \eqref{p-exp2}, by taking $\epsilon_1, \epsilon_2>0$ in (\ref{Youngt-glnn}) such that
$$
\frac{ (p-1)\chi^2}{2}(\epsilon_1+\epsilon_2)\leq \frac{b}{2}, \quad \quad \frac{ (p-1)\chi^2}{2}\epsilon_2 \leq \frac{2(p-1)}{p^2}
$$
we deduce from \eqref{lp/2-to-lp} and \eqref{Youngt-gl} that
$$
\frac{1}{p}\frac{d}{dt} \int_\Omega w^2 \leq \int_\Omega(aw^{\frac{2(p-1)}{p}}-\frac{b}{2}w^{\frac{2(p+\vartheta)}{p}})+C(p), $$
which, together with the fact
$$
\max\Bigr\{aw^{\frac{2(p-1)}{p}}-\frac{b}{2}w^{\frac{2(p+\vartheta)}{p}}+w^2:w\geq 0\Bigr\}<\infty,
$$
immediately gives that
$$
\frac{1}{p}\frac{d}{dt} \int_\Omega w^2\leq -\int_\Omega w^2+C(p) $$
for some possibly large constant $C$. The substitution of $w=u^{\frac{p}{2}}$ yields
\be\label{u^p-int-est}
\frac{1}{p}\frac{d}{dt} \int_\Omega u^p\leq -\int_\Omega u^p+C(p, r).   \ee
Solving this Gronwall inequality, we deduce that $\|u(t)\|_{L^p}$ is bounded with $p>\kappa n$  by the choice of $p$ in \eqref{p-exp2}; also, note that $p\geq 2$ by our stipulation.

Now,  the point-wise elliptic $W^{2,q}$-estimate applied to the $v$-equation in \eqref{para-elli}  with $\tau=0$ shows that $\|v(t)\|_{W^{2,p/\kappa}}$ is bounded, which is embedded in $C^1(\overline{\Omega})$ thanks to the fact  $p/\kappa>n$. As such, we can carry out the well-known Moser iteration technique \cite{Al} to obtain the $L^\infty$-bound of $u$; see details in  \cite[p. 4290-4292]{Xiang15}.

With the help of these bounds for $u$ and $v$, the extension criterion \eqref{T_m-cri} readily  implies $T_m=\infty$, and thus $u$ and $v$ are globally defined and, moreover,  $\|u(t)\|_{L^\infty}$ and  $\|v(t)\|_{W^{1,\infty}}$ are uniformly bounded  with respect to $t\in (0,\infty)$.
\end{proof}

\begin{remark}\label{para-para-try} The results obtained in Theorem \ref{thm-para-para-global-ext}  hold also for the parabolic-parabolic K-S chemotaxis-growth model  (\ref{para-elli}) with $\tau=1$: 
\begin{equation}\label{para-para}
\left\{ \begin{array}{lll}
&u_t = \nabla \cdot (\nabla u- \chi u\nabla v)+f(u),  &\quad x\in \Omega, t>0, \\
&v_t= \Delta v  -v+g(u),  &\quad x\in \Omega, t>0, \\
&\frac{\partial u}{\partial \nu}=\frac{\partial v}{\partial \nu}=0, &\quad x\in \partial \Omega, t>0,\\
&u(x,0)=u_0(x), v(x,0)=v_0(x),  &\quad x\in \Omega.
\end{array}\right.  \end{equation}
 In this case, instead of using elliptic  estimate, one uses the method of heat Neumann  semigroup (e.g. see  \cite[Lemma 4.1]{HW06}, \cite[Lemma 1]{KS08},  \cite[Lemma 1.2]{TW12} and \cite[Lemma 3.2]{Xiang15}).  Precisely, we have the following ``reciprocal" lemma:
\end{remark}
\begin{lemma}\label{expected-lemma} Let $(u, v)$ be a maximal  solution of \eqref{para-para} defined on its maximal existence interval $[0,T_m)$. If there exist  $r\in [1,\kappa n)$ and $k_1>0$ such that
$$
\|u(\cdot, t)\|_{L^r(\Omega)}\leq k_1, \quad \forall t\in (0, T_m),
$$
then
$$
\|v(\cdot, t)\|_{W^{1,q}(\Omega)}\leq C(q,r,  v_0)(1+k_1), \quad \forall t\in (0, T_m)
$$
holds for all
$$
1<q<\frac{nr}{\kappa n-r}=\frac{1}{\frac{\kappa}{r}-\frac{1}{n}}.
$$
\end{lemma}
The key to the proof of Lemma \ref{expected-lemma} is to employ the heat Neumann semigroup and  the variation-of-constants formula for $v$:
$$
v(x,t)=e^{-At}v_0(x)+\int_0^te^{-A(t-s)}g(u(x,s))ds,  \quad 0\leq t< T_m,
$$
where $e^{-At}v_0$ denotes the unique solution to the IBVP:
$$\left\{ \ba{lll}
&v_t = \Delta v- v,  &\quad x\in \Omega, t>0, \\
&\frac{\partial v}{\partial \nu}=0, &\quad x\in \partial \Omega, t>0,\\
&v(x,0)=v_0(x)\geq, \not\equiv0, &\quad x\in \Omega. \ea\right. $$

Recalling the proof of Theorem \ref{thm-para-para-global-ext}, one finds that the key ingredient therein is to derive the $L^\infty$-boundedness of $\nabla  v$. Once this is available, the Moser iteration can be used to derive the $L^\infty$-boundedness of $u$. Then, with Remark \ref{para-para-try} at hand and noticing  that $\|u(t)\|_{L^1}$ is uniformly bounded (see Lemma \ref{local-in-time}), global existence of solutions to \eqref{para-para} with  sub-linear secretion  follow directly from  Theorem \ref{thm-para-para-global-ext}.

\begin{corollary} \label{L1 to Linfty}
Let  \eqref{f-con} and \eqref{g-con} with $\kappa<2/n$ hold. Then the unique classical  global solution $(u(\cdot, t),v(\cdot, t))$  of \eqref{para-para} is uniformly bounded in $L^\infty(\Omega)\times W^{1,\infty}(\Omega)$ for all $t\in (0,\infty)$.
\end{corollary}
\begin{remark}\label{rem1}
Even in the absence of growth source, the assumption $g(u)\leq \beta u^\kappa$ with $\kappa<2/n$ may induce that $(u(\cdot, t),v(\cdot,t))$ is bounded in $L^\infty(\Omega)\times W^{1,q}(\Omega)$ for some $q>n$, cf. \cite{BBTW15, JXpre}. The point here is that  the uniform boundedness of $\|u(t)\|_{L^1}$ is sufficient to prevent  blowup of solutions. This is not usually the case as commented in the beginning of this section.
\end{remark}

It is known  from \cite{Win11} that, even for  a simpler chemotaxis-growth model  of  (\ref{para-elli}) with linear chemical production,  blowup  is still possible despite logistic dampening. Hence, it is useful to give an equivalent characterization of Theorem \ref{thm-para-para-global-ext} in terms of blowup solutions.
\begin{corollary}
Let the assumptions in Theorem \ref{thm-para-para-global-ext} hold.  Suppose that $(u,v)$ is a solution of (\ref{para-elli}) which blows up at time $t=T_m$. Then  $u$ and $v$ blow up simultaneously at $t=T_m$ in the following manner:
$$
\limsup_{t\rightarrow T_m-}\|u(\cdot, t)\|_{L^p(\Omega)}=\infty  \mbox{ for all } p>\kappa n/2 \mbox{  and  }   \limsup_{t\rightarrow T_m-}\|v(\cdot, t)\|_{W^{1,\infty}(\Omega)}=\infty.
$$
This means, for any blowup  solution $(u,v)$ of (\ref{para-elli}),  $u$ blows up not only in $L^\infty$-topology but also in $L^p$-topology for any $p>\kappa n/2$, and $v$ blows up in $W^{1,\infty}$-topology.
\end{corollary}

\section{The $L^{\kappa n/2+\epsilon}$-boundedness of $u$ and global existence}

In this section, we apply the criterion established in Theorem 3.1 to study the  most interesting minimal chemotaxis model with growth source and nonlinear production of the chemical signal:
\begin{equation}\label{para-elli-min}\left\{ \begin{array}{lll}
&u_t = \nabla \cdot (\nabla u-\chi u\nabla v)+f(u),   & \quad x\in \Omega, t>0, \\
&0= \Delta v  -v+g(u),  &\quad x\in \Omega, t>0, \\
&\frac{\partial u}{\partial \nu}=\frac{\partial v}{\partial \nu}=0, &\quad x\in \partial \Omega, t>0,\\
&u(x,0)=u_0(x)\geq, \not\equiv0,  & \quad x\in \Omega.   \end{array}\right.  \end{equation}
Here, for convenience, we restate the condtions \eqref{f-con} and \eqref{g-con}:  the growth source $f$ satisfies  $f(0)\geq 0$ and
\be\label{f-allen-con}
f(u)\leq a-bu^\theta, \quad \forall\ u\geq 0
\ee
for some $a\geq 0$,  $b>0$ and $\theta>1$,  and the nontrivial production term $g$   satisfies
\be\label{g-allen-con}
g(u)\leq \beta u^\kappa, \quad \forall\ u\geq 0
\ee
for some $\beta>0$ and $\kappa>0$.

If $\kappa < \frac{2}{n}$, then the boundedness for  \eqref{para-elli-min} even in the absence of growth source is readily ensured by Corollary \ref{L1 to Linfty}  and  Remark \ref{rem1}. Hence, we will consider  only the case  $\kappa\geq \frac{2}{n}$ in the rest of this section.

\begin{theorem}\label{thm-para-elli-min} Let $u_0\in C(\overline{\Omega}) $ and let $f$  and $g$ fulfill  \eqref{f-allen-con} and  \eqref{g-allen-con}  with
\be\label{theta-kappa>1}
\theta-\kappa>1.
\ee Then  the unique classical  global solution $(u(\cdot, t),v(\cdot, t))$  of \eqref{para-elli-min} is uniformly bounded in $L^\infty(\Omega)\times W^{1,\infty}(\Omega)$ for all $t\in (0,\infty)$.
\end{theorem}
\begin{proof} For  any  $p>1$,  we multiply  the $u$-equation in (\ref{para-elli-min}) by $pu^{p-1}$ and integrate the result over $\Omega$ by parts to deduce that
\be\label{up-est-min}
\frac{d}{dt} \int_\Omega u^p+p(p-1)\int_\Omega u^{p-2}|\nabla u|^2=(p-1)\chi p\int_\Omega  u^{p-1} \nabla u\nabla v+p\int_\Omega f(u)u^{p-1}.\ee
Testing $v$-equation in (\ref{para-elli-min}) against $u^p$, we end up with
\be\label{vp-est-min}
p\int_\Omega  u^{p-1} \nabla u\nabla v=-\int_\Omega u^{p}v+\int_\Omega  u^{p}g(u).\ee
Substituting  \eqref{vp-est-min} into \eqref{up-est-min} and using  \eqref{f-allen-con}  and \eqref{g-allen-con} yield
\be\label{up-diff-min}
\ba{ll}
&\D\frac{d}{dt} \int_\Omega u^p+\D p(p-1)\int_\Omega u^{p-2}|\nabla u|^2\\[0.25cm]
&\D \leq -(p-1)\chi\int_\Omega u^{p}v +(p-1)\beta \chi\int_\Omega  u^{p+\kappa} +ap \int_\Omega u^{p-1}-bp\int_\Omega u^{p+\theta-1}.
\ea
\ee
Thanks to the relation \eqref{theta-kappa>1}, we conclude
$$
C_p:=\max\Bigr\{(p-1)\beta \chi u^{p+\kappa} +ap u^{p-1}-bp u^{p+\theta-1}+u^p:u\geq 0\Bigr\}<\infty.
$$
Then it follows  from \eqref{up-diff-min} that
$$
\frac{d}{dt} \int_\Omega u^{p}\leq  -\int_\Omega  u^{p}+C_p, $$
which, upon a simple use of Gronwall's inequality,  directly yields that
$$
\int_\Omega u^{p}\leq \int_\Omega u_0^{p}+C_p
$$
for any $p>1$ and  for any $t\in (0,T_m)$. As a  consequence,  the criterion of Theorem \ref{thm-para-para-global-ext} immediately asserts that $T_m=\infty$ and, furthermore, $\|u(t)\|_{L^\infty}$ and $\|v(t)\|_{W^{1,\infty}}$ are uniformly bounded for $t\in (0,\infty)$.
\end{proof}

Next, we fully take the chemotatic effect into account, and extend   Theorem \ref{thm-para-elli-min}  to the borderline  case in the following way.
\begin{theorem}\label{thm-para-elli-min2}
 Let $u_0\in C(\overline{\Omega}) $ and let  $f$  and $g$ fulfill  \eqref{f-allen-con} and  \eqref{g-allen-con}  with
\be\label{theta-kappa=1}
\theta-\kappa=1.
\ee
If either
\be\label{b-con-min2}
 b> \frac{(\kappa n-2)}{\kappa n}\beta \chi
\ee
or
\be\label{b-con-min=}
 b=\frac{(\kappa n-2)}{\kappa n}\beta \chi\ \mathrm{and} \  g(u)=\beta u^\kappa,
\ee
then the unique classical  global solution $(u(\cdot, t),v(\cdot, t))$  of (\ref{para-elli-min}) is uniformly bounded in $L^\infty(\Omega)\times W^{1,\infty}(\Omega)$ for all $t\in (0,\infty)$.
\end{theorem}
\begin{proof} Due to Theorem \ref{thm-para-para-global-ext}, it suffices to   prove that $\|u(t)\|_{L^{\kappa n/2+\epsilon}}$ is uniformly bounded  for some sufficiently small $\epsilon>0$.  To this end,  we revisit \eqref{up-diff-min} by using \eqref{theta-kappa=1} to deduce
\be\label{up-diff-min2}
\frac{d}{dt} \int_\Omega u^p\leq -(p-1)\chi\int_\Omega u^{p}v -[bp-(p-1)\beta \chi]\int_\Omega  u^{p+\kappa}+ap \int_\Omega u^{p-1}.
\ee
Let us first treat the strict inequality case of \eqref{b-con-min2}; that is
$$
b> \frac{(\kappa n-2)}{\kappa n}\beta \chi=\frac{(\frac{\kappa n}{2}-1)}{\frac{\kappa n}{2}}\beta \chi.
$$
This allows us to fix a small  $\epsilon>0$ in such a way that
\be\label{bepi-con-min}
b>\frac{[(\frac{\kappa n}{2}+\epsilon)-1]}{(\frac{\kappa n}{2}+\epsilon)}\beta \chi.
\ee
Setting  $p=\kappa n/2+\epsilon$ and using \eqref{bepi-con-min}, we see  $[bp-(p-1)\beta \chi]>0$. The fact $\kappa>0$ then trivially gives  that
$$
C=\max\Bigr\{-[bp-(p-1)\beta \chi]u^{p+\kappa}+apu^{p-1}+u^p: u\geq 0\Bigr\}<\infty,
$$
which combined with \eqref{up-diff-min2} leads us to
$$
\frac{d}{dt} \int_\Omega u^{\kappa n/2+\epsilon}\leq  -\int_\Omega  u^{\kappa n/2+\epsilon}+C(n,\epsilon, \beta, \chi, a, b, \kappa, \Omega).$$
This immediately shows that $\|u(t)\|_{L^{\kappa n/2+\epsilon}}$ is uniformly bounded for $t\in (0,T_m)$.

Let us now examine the borderline case of \eqref{b-con-min=}; that is,
\be\label{b-con-min3}
b=\frac{(\frac{\kappa n}{2}-1)}{\frac{\kappa n}{2}}\beta \chi, \quad g(u)=\beta u^\kappa.
\ee
We wish to find a finite $C=C( \beta,\chi, a, b, \kappa, p, \Omega)>0$ such that
\be\label{wish-p0}
-(p-1)\chi\int_\Omega u^{p}v -[bp-(p-1)\beta \chi]\int_\Omega  u^{p+\kappa}+ap \int_\Omega u^{p-1}\leq -\int_\Omega u^p+C
\ee
for some $p>\kappa n/2$.  If \eqref{wish-p0} is valid, then we get from \eqref{up-diff-min2} that
$$
\frac{d}{dt} \int_\Omega u^p\leq -\int_\Omega u^p+C,
$$
which gives rises to the desired boundedness of $\|u(t)\|_{L^p}$.

We need only to consider this inequality for  $t \nearrow T_m$ (meaning $t<T_m$ but sufficiently close to $T_m$).  To show \eqref{wish-p0}, we denote by
$$
\Omega_\infty=\Bigr\{z\in\overline{\Omega}: u(z,t)=\sup_{x\in\Omega}u(x,t), u(z,t) \text{ becomes unbounded for } t\nearrow T_m\Bigr\}
$$
and
$$
\Omega_0=\Bigr\{z\in\overline{\Omega}: \frac{v(z,t)}{u^\kappa(z,t)} =\inf_{x\in\Omega}\frac{v(x,t)}{u^\kappa(x,t)}, \frac{v(z,t)}{u^\kappa(z,t)} \text{ becomes arbitrarily small for } t\nearrow  T_m\Bigr\}.
$$
Then it follows easily that both $\Omega_\infty$ and $\Omega_0$ are closed.

Now, dividing the $v$-equation in (\ref{para-elli-min}) by $v$ and then integrating the result over   $\Omega$ by parts, we get
\be\label{divide-v2}
\int_\Omega  \frac{1}{v^2}|\nabla v|^2+\beta \int_\Omega \frac{u^\kappa}{v}=|\Omega|.\ee
For any $t\in  (0, T_m)$ and $t\approx T_m$, from
$$
|\Omega|\geq \beta \int_\Omega \frac{u^\kappa}{v}\geq \int_{\Omega_\infty\cap\Omega_0\cap \Omega}\frac{u^\kappa}{v}\geq  \inf_{z\in \Omega_\infty\cap\Omega_0\cap \Omega} \frac{u^\kappa(z,t)}{v(z,t)}  | \Omega_\infty\cap\Omega_0\cap \Omega|,
$$
it gives that $|\Omega_\infty\cap\Omega_0\cap \Omega|=0$, where the closedness of $\Omega_\infty$ and $\Omega_0$ may be utilized. As a result, the integration over $\Omega$ can be split into the following way:
\be\label{omega-split}
\int_\Omega=\int_{\Omega\setminus(\Omega_\infty\cap\Omega_0\cap\Omega)} =\int_{(\Omega_\infty\cap\Omega)\cap(\Omega\setminus\Omega_0)}+
\int_{\Omega\setminus\Omega_\infty}.
\ee
On $\Omega\setminus\Omega_\infty$ (if nonempty), we have $u(\cdot,t)$ is uniformly bounded and thus $v(\cdot,t)$ is also uniformly bounded for $t<T_m$ due to  the elliptic estimate applied to the $v$-equation in \eqref{para-elli-min}.  This entails that
\be\label{wish-p1}
\int_{\Omega\setminus\Omega_\infty} \Bigr(-(p-1)\chi u^pv -[bp-(p-1)\beta \chi] u^{p+\kappa}+ap  u^{p-1}+ u^p\Bigr)\leq C.
\ee
Next, we shall assume, without loss of generality, that $(\Omega_\infty\cap\Omega)\cap(\Omega\setminus\Omega_0)\neq \emptyset$. Then,
under \eqref{b-con-min3}, we wish  to show that
\be\label{wish-p2}-\int_{(\Omega_\infty\cap\Omega)\cap(\Omega\setminus\Omega_0)} u^{p+\kappa}\bigg((p-1)\chi \frac{v}{u^\kappa} +[bp-(p-1)\beta \chi] -\frac{ap}{u^{\kappa+1}} - \frac{1}{u^\kappa}\bigg)\leq 0
\ee
holds for  $p=\kappa n/2+\epsilon$ with some $\epsilon>0$ sufficiently small. Note that   \eqref{b-con-min3} implies $bp-(p-1)\beta \chi<0$ for any $p>\kappa n/2$. Then the desired inequality \eqref{wish-p2} will hold if
\be\label{wish-p3}
(p-1)\chi \frac{v}{u^k} +[bp-(p-1)\beta \chi] > 0  \quad \text{ on } \quad (\Omega_\infty\cap\Omega)\cap(\Omega\setminus\Omega_0).
\ee
From \eqref{b-con-min3}, one can see (\ref{wish-p3}) is true if
\be\label{wish-p4}
(p-1)\chi \sigma +[bp-(p-1)\beta \chi] > 0  \Longleftrightarrow (\sigma-\frac{2\beta}{\kappa n})p> (\sigma-\beta),
\ee
where
$$
\sigma:=\inf_{x\in(\Omega_\infty\cap\Omega)\cap(\Omega\setminus\Omega_0)}\frac{v(x,t)}{u^\kappa(x,t)}>0,\quad  t\approx T_m.
$$
If $\sigma> 2\beta/(\kappa n)$, it is fairly easy to see  \eqref{wish-p4} holds for any
$$
p>\frac{\kappa n(\sigma-\beta)}{\kappa n\sigma - 2\beta}.
$$
If $\sigma=2\beta/(\kappa n)$, then \eqref{wish-p4} is true for any $p>0$ by noticing the fact that $\kappa n>2$.  Lastly, if $\sigma<2\beta/(\kappa n)$, the inequality  \eqref{wish-p4} is true for
$$
p< \frac{\kappa n(\beta-\sigma)}{2\beta-\kappa n\sigma}.
$$
While,  the fact  $\kappa n>2$ entails that
$$
\frac{\kappa n(\beta-\sigma)}{2\beta-\kappa n\sigma}>\frac{\kappa}{2}n,
$$
which implies  that \eqref{wish-p4} holds  for $p=\kappa n/2+\epsilon$ with some $\epsilon>0$ sufficiently small.

Finally,   \eqref{wish-p0} is obtained by the combination of  \eqref{wish-p1}, \eqref{wish-p2} and \eqref{omega-split}.
\end{proof}

\begin{remark}Tello and Winkler \cite{TW07} considered the parabolic-elliptic minimal chemotaxis model with logistic source and linear production of the chemical signal:
\begin{equation}\label{para-elli-min-log}
\left\{ \begin{array}{lll}
&u_t = \nabla \cdot (\nabla u-\chi u\nabla v)+f(u),  &\quad x\in \Omega, t>0, \\[0.2cm]
&0= \Delta v -v+u,  &\quad x\in \Omega, t>0, \\[0.2cm]
&\frac{\partial u}{\partial \nu}=\frac{\partial v}{\partial \nu}=0, &\quad x\in \partial \Omega, t>0,\\[0.2cm]
&u(x,0)=u_0(x)\geq, \not\equiv0, &\quad x\in \Omega, 
\end{array}\right.  \end{equation}
where $f$ satisfies the logistic condition: $f(0)\geq 0$ and
\be\label{f-log-con}
f(u)\leq a-bu^2, \quad \forall u\geq 0
\ee
for some $a\geq 0$,  $b>0$. Therein, under \eqref{f-log-con},  they showed, if
$$
b>\frac{(n-2)}{n}\chi,
$$
then the unique solution  $(u, v)$ of \eqref{para-elli-min-log} is uniformly bounded in  $\Omega\times (0,\infty)$. Here, we mention that their proof involves heavy knowledge about Neumann heat semi-group and is a little lengthy. With $\theta=2$ and $\kappa=1$, our Theorem \ref{thm-para-elli-min2} not only extends their work \cite{TW07} to the equality case $b=(n-2)/n\chi$ (which may suggest this number is not critical in this respect) but also provides a simpler semigroup-free proof for this model.
\end{remark}
\begin{remark}From the discussions in Section 3 and the work of \cite{Win08} on sub-quadratic dampening enforcing the existence of  global "very weak" solutions, we are led to speculate that no blow-up would occur for the minimal-chemotaxis-growth model \eqref{para-elli-min} whenever
$$
\theta-\kappa>1-\frac{1}{n}.
$$
If this turns out to be true, then it is a significant improvement of Theorems  \ref{thm-para-elli-min}  and \ref{thm-para-elli-min2} and hence of existing results (cf. \cite{TW07, CZ13, WMZ14, Zh15}). In particular,  under additional smallness assumptions, this has been verified  for  \eqref{para-elli-min-log} with $f$ satisfying  $f(u)\leq a-bu^\theta$ for all $u\geq 0$ and for some $a\geq 0$, $b>0$ and
$$
\theta> 2-\frac{1}{n}.
$$ Our approach above does not enable us to obtain such a sharp result. Innovative ways should  be detected  to either prove or disprove this challenging question.
\end{remark}
\section{Steady states for the K-S model}

In this section, we study  the steady states to the chemotaxis model (\ref{para-elli-min}):
\begin{equation}\label{para-elli-min-ss}\left\{ \begin{array}{lll}
&0 = \nabla \cdot (\nabla u-\chi u\nabla v)+f(u),  \quad & x\in \Omega, \\[0.2cm]
&0= \Delta v -v+g(u),  \quad & x\in \Omega, \\[0.2cm]
&\frac{\partial u}{\partial \nu}=\frac{\partial v}{\partial \nu}=0, \quad & x\in \partial \Omega.   \end{array}\right.  \end{equation}

First of all, some {\it a priori} estimates  and regularity results for the solution of \eqref{para-elli-min-ss} are  needed in the subsequent discussions.
\begin{lemma}\label{pre-est-lem} Let $f$ and $g$ satisfy \eqref{f-allen-con} and \eqref{g-allen-con} with $\theta>\kappa$, and $(u,v)$ be a positive solution of  \eqref{para-elli-min-ss}. Then
\be\label{pre-est}
\int_\Omega u^\theta\leq \frac{a}{b}|\Omega|, \quad \quad \min_{\bar{\Omega}}u\leq K,   \quad\quad  \int_\Omega v\leq \frac{1}{\beta}\Bigr(\frac{a}{b}\Bigr)^{\frac{\kappa}{\theta}}|\Omega|,
\ee
where $K$ is the largest zero point of $f$. Moreover, the $W^{2, \frac{\theta}{\kappa}}$-norm of $v$ is uniformly bounded in $\chi$. In particular, if $f(u)=c u-bu^\theta$, then
$\max_{\bar{\Omega}} u\geq K=(c/b)^{(\theta-1)^{-1}}$.
\end{lemma}
\begin{proof}Integrating the $u$-equation and using the fact $f(u)\leq a-bu^\theta$, we have
$$
0=\int_\Omega f(u)\leq \int_\Omega (a-bu^\theta),
$$
which directly gives the first two inequalities in \eqref{pre-est}. Then integrating  the $v$-equation, using $g(u)\leq \beta u^\kappa$ and H\"{o}lder inequality,  we arrive at  the last desired inequality in \eqref{pre-est}.

Notice that
$$
\|g(u)\|_{L^{\frac{\theta}{\kappa}}}\leq \beta\|u^\kappa\|_{L^{\frac{\theta}{\kappa}}}= \beta\Bigr(\int_\Omega u^\theta\Bigr)^{\frac{\kappa}{\theta}}\leq\beta \Bigr(\frac{a}{b}|\Omega|\Bigr)^{\frac{\kappa}{\theta}},
$$
then the  elliptic regularity applied to the $v$-equation in  \eqref{para-elli-min-ss} gives the stated $W^{2, \frac{\theta}{\kappa}}$-estimate for $v$. Especially, for $f(u)=cu-bu^\theta$, if $\max_{\bar{\Omega}}u<K$, then $f(u)>0$ on $\Omega$ and so $\int_\Omega f(u)>0$, which is a contradiction.
\end{proof}

Using the similar (test function)  argument  as done in \cite{TW07}, we derive some regularity results for  \eqref{para-elli-min-ss}, showing the solution typically will never become singular (i.e., unbounded).
\begin{lemma}\label{nonsingular-chi} Let $f$ and $g$ satisfy \eqref{f-allen-con} and \eqref{g-allen-con} with
\be\label{theta-kpppa=1}
\theta-1\geq \kappa
\ee
 and let $(u,v) $ be a nonnegative solution of \eqref{para-elli-min-ss}.
\begin{itemize}
\item[(i)] $u\in L^{p+\kappa}(\Omega)$  and $v\in  L^{q+1}(\Omega)$ for any  $p<\frac{\beta \chi}{(\beta\chi-b)^+}$ and $q<\frac{\beta \chi}{\kappa (\beta\chi-b)^+}$;
\item[(ii)] for  $\theta-1> \kappa$,    $u, v$ is  bounded in $L^\infty(\Omega)$ and $u,v\in C^{1+\gamma}(\bar{\Omega})$ for all $\gamma\in (0,1)$;
\item[(iii)] for  $\theta-1= \kappa$, if
\be\label{parameter-con}
n\leq \frac{2}{\kappa}+2 \text{   or   }   \Bigr\{ n> \frac{2}{\kappa}+2, \quad b>\frac{[(n-2)\kappa-2]}{(n-2)\kappa}\beta \chi\Bigr\},
\ee
then $u, v$ is  bounded in $L^\infty(\Omega)$ and $u,v\in C^{1+\gamma}(\bar{\Omega})$ for all $\gamma\in (0,1)$;
\item[(iv)]  if $f(u)>0$ on $(0, (\frac{a}{b})^{\frac{1}{\theta}})$, then any solution of \eqref{para-elli-min-ss} satisfies
\be\label{u-bdd-v}
(\frac{a}{b})^{\frac{1}{\theta}}e^{\chi ( \min_{\bar{\Omega}}v-\max_{\bar{\Omega}}v)}\leq u\leq (\frac{a}{b})^{\frac{1}{\theta}}e^{\chi ( \max_{\bar{\Omega}}v-\min_{\bar{\Omega}}v)}.
\ee
In particular, if $v$ is bounded, then $u,v\in C^{1+\gamma}(\bar{\Omega})$ for all $\gamma\in (0,1)$.
\end{itemize}
\end{lemma}
\begin{proof}
(i) The elliptic counterpart of \eqref{up-diff-min} is
 \be\label{up-diff-min-ellip}
 p(p-1)\int_\Omega u^{p-2}|\nabla u|^2+(p-1)\chi\int_\Omega u^{p}v -(p-1)\beta \chi\int_\Omega  u^{p+\kappa} +bp\int_\Omega u^{p+\theta-1}\leq ap \int_\Omega u^{p-1}.
\ee
In the case $\theta-1>\kappa$,  an easy application of Young's inequality with $\epsilon$ to  \eqref{up-diff-min-ellip} gives that $\int_\Omega u^{p+\theta-1}$ is bounded for any $p>1$; while, in the case $\theta-1=\kappa$, it  follows from \eqref{up-diff-min-ellip}  that
$$
[bp-(p-1)\beta \chi]\int_\Omega  u^{p+\kappa}\leq ap \int_\Omega u^{p-1},
$$
which immediately implies   $u\in L^{p+\kappa}(\Omega)$ for any  $p<\frac{\beta \chi}{(\beta\chi-b)^+}$. Then multiplying  the $v$-equation by $v^q$,  integrating by parts and using  \eqref{g-allen-con} and Young's inequality, we deduce
$$
q\int_\Omega v^{q-1}|\nabla v|^2+\int_\Omega v^{q+1}\leq \beta \int_\Omega u^\kappa v^q\leq  \frac{1}{2}\int_\Omega v^{q+1}+\frac{\beta}{q+1}(\frac{2\beta q}{q+1})^q\int_\Omega u^{(q+1)\kappa},
$$
which coupled with the integrability of $u$  yields that $v\in L^{q+1}(\Omega)$ for any  $q<\frac{\beta \chi}{\kappa (\beta\chi-b)^+}$.

(ii) Let us use the $v$-equation in \eqref{para-elli-min-ss} to rewrite  the system  \eqref{para-elli-min-ss} as
\begin{equation}\label{rewrite u-v}\left\{ \begin{array}{lll}
& \Delta u-\chi \nabla u\nabla v=\chi u(v-g(u))-f(u),  \quad & x\in \Omega, \\[0.2cm]
&\Delta v =v-g(u),  \quad & x\in \Omega,  \\[0.2cm]
&\frac{\partial u}{\partial \nu}=\frac{\partial v}{\partial \nu}=0, \quad & x\in \partial \Omega.   \end{array}\right.  \end{equation}

Let $x_1$ and $x_2$ be the maximum points of $u$ and $v$ in $\overline{\Omega}$, respectively.  We then apply  the  Hopf lemma and maximum principle, cf.  \cite[Lemma 2.1] {LN99}  and use  \eqref{f-allen-con} and \eqref{g-allen-con}   to get
$$\left\{ \ba{ll}
&  \chi u(x_1)(v(x_1)-\beta u^\kappa (x_1))\leq \chi u(x_1)(v(x_1)-g(u(x_1)))\leq f(u(x_1))\leq a-bu^\theta(x_1), \\[0.25cm]
& v(x_2)\leq g(u(x_2))\leq \beta u^\kappa(x_2),
 \ea\right. $$
from which it follows that
\be\label{mp-bdd}\left\{ \ba{ll}
&  \chi u(x_1)(v(x_1)-\beta u^\kappa (x_1)+\frac{b}{\chi}u^{\theta-1}(x_1))\leq a, \\[0.25cm]
& v(x_2)\leq g(u(x_2))\leq \beta u^\kappa(x_2)\leq \beta u^\kappa(x_1).
 \ea\right. \ee
If $\theta-1>\kappa$, then the boundedness of $u$ and $v$ follows from \eqref{mp-bdd}. Then the desired  statement that $u,v\in C^{1+\gamma}(\bar{\Omega})$ for all $\gamma\in (0,1)$ follows from the standard elliptic regularity, cf. \cite[P. 868, Step 4]{TW07}.

(iii) The $W^{2,p}$ elliptic regularity applied to
$$
-\Delta v+v=g(u)\in L^{\frac{p}{\kappa}+1}(\Omega), \quad \forall p<\frac{\beta \chi}{(\beta\chi-b)^+}
$$
shows that $v\in W^{2,\frac{p}{\kappa}+1}(\Omega)$. Notice that \eqref{parameter-con} implies
$$
\frac{2\beta \chi}{\kappa (\beta\chi-b)^+}+2>n.
$$
Then the Sobolev embedding says  $v\in L^\infty(\Omega)$.

To show the boundedness of $u$, we rewrite the  $u$-equation  as
\be\label{rewrite-u}
-\nabla\cdot(e^{\chi v} \nabla(u e^{-\chi v}))=f(u).
\ee
Testing it against $(u e^{-\chi v}-(\frac{a}{b})^{\frac{1}{\theta}})_+$ and using \eqref{f-allen-con}, we  end up with
$$
\int_\Omega e^{\chi v} |\nabla(u e^{-\chi v}-(\frac{a}{b})^{\frac{1}{\theta}})_+|^2=\int_\Omega f(u)(u e^{-\chi v}-(\frac{a}{b})^{\frac{1}{\theta}})_+=\int_{\{u>(\frac{a}{b})^{\frac{1}{\theta}}\}}f(u)(u e^{-\chi v}-(\frac{a}{b})^{\frac{1}{\theta}})_+\leq 0,
$$
which implies
\be\label{u-bdd by v}
u\leq (\frac{a}{b})^{\frac{1}{\theta}} e^{\chi v},
\ee
giving the desired bound for $u$.

(iv) Let $(u,v)$ be a  solution of \eqref{para-elli-min-ss}. Then we test \eqref{rewrite-u} by
$$
\Bigr(ue^{-\chi v}-(\frac{a}{b})^{\frac{1}{\theta}}e^{-\chi \max_{\bar{\Omega}}v}\Bigr)_-
$$
to derive
$$
\int_\Omega f(u)\Bigr(ue^{-\chi v}-(\frac{a}{b})^{\frac{1}{\theta}}e^{-\chi \max_{\bar{\Omega}}v}\Bigr)_-=\int_\Omega e^{\chi v} \nabla(u e^{-\chi v})\nabla \Bigr(ue^{-\chi v}-(\frac{a}{b})^{\frac{1}{\theta}}e^{-\chi \max_{\bar{\Omega}}v}\Bigr)_-\leq 0,
$$
which deduces
$$
\int_{\{u<(\frac{a}{b})^{\frac{1}{\theta}}\}} f(u)e^{-\chi v} \Bigr(u-(\frac{a}{b})^{\frac{1}{\theta}}e^{\chi(v- \max_{\bar{\Omega}}v)}\Bigr)_-=\int_\Omega f(u)e^{-\chi v} \Bigr(u-(\frac{a}{b})^{\frac{1}{\theta}}e^{\chi(v- \max_{\bar{\Omega}}v)}\Bigr)_-\leq 0.
$$
Since $f$ is positive on $(0, (\frac{a}{b})^{\frac{1}{\theta}})$, it follows that
$$
\Bigr(u-(\frac{a}{b})^{\frac{1}{\theta}}e^{\chi(v- \max_{\bar{\Omega}}v)}\Bigr)_-\equiv 0 \Longrightarrow u\geq  (\frac{a}{b})^{\frac{1}{\theta}}e^{\chi(\min_{\bar{\Omega}}v- \max_{\bar{\Omega}}v)}.
$$
In a similar  way, testing \eqref{rewrite-u} by
$$
\Bigr(ue^{-\chi v}-(\frac{a}{b})^{\frac{1}{\theta}}e^{-\chi \min_{\bar{\Omega}}v}\Bigr)_+
$$
yields the right inequality in \eqref{u-bdd-v}.
\end{proof}

Next, we investigate the capability of the system \eqref{para-elli-min-ss} to form patterns. We perform  Leray-Schauder index formula (The  possibility of realization of such method was mentioned in \cite{KOST12} but not carried out for a simpler model)  to show that, for each equilibrium state,  the  stationary  system \eqref{para-elli-min-ss} admits an increasing  sequence of $\{\chi_k\}_{k=1}^\infty$ such that it has at least one nonconstant solution whenever $\chi\in (\chi_{2k-1},\chi_{2k}), k=1,2, \cdots$. More precisely, we have the following theorem.
\begin{theorem}\label{bif-thm} Let $Z=\{u_0>0|u_0 \mbox{ is an isolated zero point of } f \mbox{ and } g^\prime(u_0)>0\}$. Then,  for each $u_0\in Z$,  there exists a positive increasing sequence $\{\chi_k=\chi_k(u_0)\}$ with the property
$$
0<\chi_1(u_0)<\chi_2(u_0)<\cdots<\chi_k(u_0)<\chi_{k+1}(u_0)\rightarrow \infty
$$
such that, for every
$$
\chi\in \bigcup_{u_0\in Z} \bigcup_{k=1}^\infty\Bigr( \chi_{2k-1}(u_0), \chi_{2k}(u_0)\Bigr):=P_\chi,
$$
the stationary chemotaxis-growth system \eqref{para-elli-min-ss} has at least one nonconstant  solution.
\end{theorem}
Before presenting the proof, we want to remark that   Theorem \ref{bif-thm} not only gives the existence of non-constant solutions of  \eqref{para-elli-min-ss} which is a generalization of the model considered by Tello and Winkler \cite{TW07} where $f$ is of logistics type, but also provides more explicit existence conditions which is easier to verify. Our proof is the consequence of bifurcation from "eigenvalues" of odd multiplicity.
 \begin{proof}[Proof of Theorem \ref{bif-thm}] By the $v$-equation, \eqref{para-elli-min-ss} is equivalent to
 \begin{equation} \label{equiv-ss} \left\{ \ba{lll}
& -\Delta u+\chi\nabla u \nabla v=-\chi u[v-g(u)]+f(u),  &\quad \mbox{in }  \Omega,\\
&-\Delta v + v=g(u), &\quad \mbox{in }  \Omega,\\
&\frac{\partial u}{\partial \nu}=\frac{\partial v}{\partial \nu}=0,   &\quad \mbox{on } \partial \Omega. \\
 \ea\right.\end{equation}
 Linearizing the system \eqref{equiv-ss} about  a generic equilibrium state $(u_0,v_0):=(u_0,g(u_0))$ with  $u_0\in Z$, we arrive at  the linearized system
 \begin{equation}\label{lin-eq}\left\{ \ba{lll}
& -\Delta  u+ u=[\chi g^\prime (u_0)u_0+f^\prime(u_0)+1]u-\chi u_0 v,  &\quad \mbox{in }  \Omega,\\
&-\Delta v  + v=g^\prime(u_0)u, &\quad \mbox{in }  \Omega,\\
&\frac{\partial u}{\partial \nu}=\frac{\partial v}{\partial \nu}=0,   &\quad \mbox{on } \partial \Omega. \\
 \ea\right. \end{equation}
 Let
 $$
 A(\chi)=\left(
     \begin{array}{cc}
       g^\prime (u_0)u_0\chi+f^\prime(u_0)+1& -\chi u_0 \\
       g^\prime(u_0) & 0 \\
     \end{array}
   \right).
 $$
 By direct computations,  the eigenvalues of $A$ are
  \begin{equation}\label{lambda-pm}\ba{ll}
&\lambda^{\pm}(\chi)\\[0.2cm]
&=\D\frac{1}{2}\Bigr[g^\prime (u_0)u_0\chi+f^\prime(u_0)+1\pm \sqrt{[g^\prime (u_0)u_0\chi+f^\prime(u_0)-1]^2+4f^\prime (u_0)}\Bigr].
\ea \end{equation}
 Case 1: $f^\prime(u_0)\geq 0$. In this case,  $\lambda^\pm(\chi)$ are  defined for all $\chi>0$ and  are increasing with
$$
0<\lambda^-(\chi)<\lambda^-(\infty)=1\leq f^\prime(u_0)+1<\lambda^+(\chi)<\lambda^+(\infty)= \infty .
$$
Case 2: $f^\prime(u_0)< 0$. In this case,  $\lambda^\pm(\chi)$ are  defined for all
$$
\chi\geq \frac{1+2\sqrt{-f^\prime(u_0)}-f^\prime(u_0)}{g^\prime(u_0)u_0}:=\hat{\chi}_0,
$$
where $\lambda^+(\chi)$ is   increasing and $\lambda^-(\chi)$ is decreasing for all $\chi \geq \hat{\chi}_0$ with
$$
\lambda^-(\infty)<\lambda^-(\chi)<\lambda^-(\hat{\chi}_0)=1+\sqrt{-f^\prime(u_0)}< \lambda^+(\chi)<\lambda^+(\infty)= \infty .
$$
It is well-known that  the eigenvalue problem
$$
-\Delta u+u=\sigma u  \quad \text{  in }  \Omega, \quad \quad \frac{\partial u}{\partial \nu}=0 \text{  on } \partial\Omega
$$
has a sequence of eigenvalues $\{\sigma_k\}_{k=0}^\infty$ with $1=\sigma_0<\sigma_1<\cdots<\sigma_k<\cdots \rightarrow \infty$ and the collection of their eigenfunctions $\{e_k(x)\}_{k=0}^\infty$  forms a complete orthogonal basis for $L^2(\Omega)$.

In case of $f^\prime(u_0)\geq 0$, we have  $\{\lambda^-(\chi):\chi>0\}\cap \Sigma=\emptyset$, hereafter $\Sigma=\{\sigma_j: j=0,1,2, \cdots\}$. In the case of $f^\prime(u_0)< 0$, we fix a $\tilde{\chi}_1$ according to
$$
\tilde{\chi}_1=\left\{ \ba{lll}
& \frac{1}{\lambda^-(\sigma_1)},  & \mbox {if  }  \sigma_1<1+\sqrt{-f^\prime(u_0)}, \\[0.2cm]
&\text{ any number} \geq\hat{\chi}_0, & \mbox {if  }  \sigma_1\geq 1+\sqrt{-f^\prime(u_0)}.
 \ea\right.
$$
Since $\lambda^-(\chi)$ is decreasing and $\lambda^-(\chi)>1=\sigma_0$,  we have   $\{\lambda^-(\chi):\chi>\tilde{\chi}_1\}\cap \Sigma=\emptyset$.

 Let $\hat{\chi}_1=0$ if $f^\prime(u_0)\geq 0$  and $\hat{\chi}_1=\tilde{\chi}_1$ if $f^\prime(u_0)< 0$.   For any  eigenvalue  $\sigma_k\in \Sigma\cap \{\lambda^+(\chi): \chi>\hat{\chi}_1\}$,   we set  $\hat{\chi}_k=(\lambda^+)^{-1}(\sigma_k)$, which is well-defined by the properties of $\lambda^+$. Then it  follows readily from the properties of $\lambda^\pm$  that
\be\label{lambda-pm-Sig}
\lambda^+(\hat{\chi}_k)\in \Sigma, \quad \quad \lambda^-(\hat{\chi}_k)\not\in \Sigma.
\ee
 Choose an open neighborhood $O_k=(\lambda^+)^{-1}((\sigma_{k-1},\sigma_{k+1}))\cap(\hat{\chi}_1,\infty)$  of $\hat{\chi}_k$ such that  for any $\sigma\in \Sigma$ with $\sigma\neq \lambda^+(\hat{\chi}_k)$, we have $\sigma\neq \lambda^\pm(\chi)$ for any $\chi\in O_k$. That is, $(\lambda^+(O_k)\cup \lambda^-(O_k))\cap \Sigma=\{\lambda^+(\hat{\chi}_k)\}$. We now consider the subsets
 $$
 O_k^+=\{\chi\in O_k:\lambda^+(\chi)>\lambda^+(\hat{\chi}_k)\}=(\lambda^+)^{-1}((\sigma_{k},\sigma_{k+1}))\cap(\hat{\chi}_1,\infty)=(\hat{\chi}_k, \hat{\chi}_{k+1}),
 $$
 $$
 O_k^-=\{\chi\in O_k:\lambda^+(\chi)<\lambda^+(\hat{\chi}_k)\}=(\lambda^+)^{-1}((\sigma_{k-1},\sigma_{k}))\cap(\hat{\chi}_1,\infty)=(\hat{\chi}_{k-1}, \hat{\chi}_{k}),
 $$
and the  space $X=\{u\in C^1(\Omega): \frac{\partial u}{\partial \nu}=0 \text{  on  } \partial \Omega\}$.  Let  $L_k^\pm: X^2\rightarrow X^2$  be defined by
$$
L_k^\pm=I-(-\Delta +I)^{-1}A(\chi), \quad \quad \chi\in O_k^\pm,
$$
where the compact operator $(-\Delta +I)^{-1}:X^2\rightarrow X^2$ is the inverse of $-\Delta +I$ in $X^2$.

In the sequel,  we will show that $(0,0)\not\in L_k^\pm(\partial B((0,0),r_k^\pm)) $ for any $r_k^\pm>0$, where $B((u_0,v_0),r)$ denotes  the open ball in $X^2$ centered at $(u_0,v_0)$ with  radius $r$. Indeed, suppose not, then \eqref{lin-eq} has a nontrivial solution $(u,v)$.  Let
$$
U_j=\int_\Omega ue_j, \quad\quad V_j=\int_\Omega ve_j.
$$
Then multiplying \eqref{lin-eq} by $e_j$ and integrating over $\Omega$, we get an algebraic system for $U_j$ and $V_j$:
$$\left\{ \ba{ll}
& [\sigma_j-g^\prime (u_0)u_0\chi-f^\prime(u_0) -1]U_j+\chi u_0V_j=0, \\[0.25cm]
&-g^\prime(u_0)U_j+\sigma_k V_j=0,
\ea\right.$$
which has a nonzero solution $(U_j, V_j)$ for some $j$ if and only if
 \be\label{sigmaj-egi}
 \sigma_j^2-[g^\prime (u_0)u_0\chi+f^\prime(u_0) +1]\sigma_j+g^\prime(u_0)u_0\chi=0.
 \ee
 Solving \eqref{sigmaj-egi} and comparing \eqref{lambda-pm}, we discover that $\sigma_j=\lambda^\pm(\chi)$, which contradicts the fact that  $(\lambda^+(O_k^\pm)\cup\lambda^-(O_k^\pm))\cap \Sigma=\emptyset$. Hence,  $(U_j, V_j)=(0,0)$  for all $j\geq 0$, and so \eqref{lin-eq} has only the zero solution by the completeness of eigenfunctions. This tells us that  $(0,0)\not\in L_k^\pm(\partial B((0,0),r_k^\pm)) $ or, equivalently,   1 is not an eigenvalue  of $(-\Delta +I)^{-1}A(\chi)$ for $\chi\in O_k^\pm$.

Then being a compact perturbation of identity,  the L-S degree,  $\deg(L^\pm_k, B((0,0),r_k^\pm), \cdot)$, is well-defined, and
$$
\deg(L_k^+, B((0,0),r_k^+), (0,0))=-\deg(L_k^-, B((0,0),r_k^-),(0, 0)),
$$
 cf. Smoller \cite{Smoller83}. In light of \eqref{equiv-ss}, we consider the nonlinear operator defined by
$$
 h(u,v;\chi)=\left(
     \begin{array}{cc}
      - \chi\nabla u \nabla v-\chi u[v-g(u)]+f(u)+u\\
      g(u) \\
     \end{array}
   \right).
 $$
Notice that, for $\chi\in O_k^\pm$ and $r_k^\pm$ small enough, the operator
$$
H_k^\pm(u,v;\chi)=I-(-\Delta +I)^{-1}(h(u,v;\chi))
$$ is a continuous and compact perturbation of the identity in $B((u_0,v_0),r_k^\pm)$ . Moreover, for such small $r_k^\pm$, we have
\be\label{deg-lk-hk}
\deg(L_k^\pm, B((0,0),r_k^\pm), (0,0))=\deg(H_k^\pm, B((u_0,v_0),r_k^\pm), (0,0)).
\ee
To calculate  $\deg(L_k^\pm, B((0,0),r_k^\pm), 0)$, we invoke the Leray-Schauder index formula, cf.  Nirenberg \cite[Theorem 2.8]{Nirenberg74}.  To this end, we need to ensure  that $1$ is not an eigenvalue of $(-\Delta +I)^{-1}A(\chi)$ for $\chi\in O_k^\pm$. This indeed has been shown above; while, for later usage, we give its proof via eigen-expansion.

 By definition, $((u,v),\mu)$ is an eigen-pair of $(-\Delta +I)^{-1}A(\chi)$ if and only if
\be\label{eigen-pair}\left\{ \ba{lll}
& -\mu \Delta  u+ \mu u=[g^\prime (u_0)u_0\chi+f^\prime(u_0)+1]u-\chi u_0 v,  &\quad \mbox{  in  }  \Omega,\\[0.2cm]
&-\mu\Delta v  + \mu v=g^\prime(u_0)u, &\quad  \mbox{ in }  \Omega,\\[0.2cm]
&\frac{\partial u}{\partial \nu}=\frac{\partial v}{\partial \nu}=0,   &\quad \mbox{ on } \partial \Omega. \\[0.2cm]
 \ea\right. \ee
By the idea of eigen-expansion, we let
 \be\label{eigen-de}
 u(x)=\sum_{j=0}^\infty u_j e_j(x), \quad \quad \quad  v(x)=\sum_{j=0}^\infty v_je_j(x).
 \ee
Substituting \eqref{eigen-de} into \eqref{eigen-pair} and using the completeness of eigenfunctions $\{e_j\}$, we obtain  an algebraic system in $u_j$ and $v_j$ as follows.
$$\left\{ \ba{ll}
& [\sigma_j\mu-g^\prime (u_0)u_0\chi-f^\prime(u_0) -1]u_j+\chi u_0v_j=0, \\[0.25cm]
&-g^\prime(u_0)u_j+\sigma_j\mu v_j=0,
\ea\right.$$
which has a nonzero solution $(u_j,v_j)$ for some $j$ if and only if
 \be\label{muj-egi}
 \sigma_j^2\mu^2-[g^\prime (u_0)u_0\chi+f^\prime(u_0) +1]\sigma_j\mu+g^\prime(u_0)u_0\chi=0.
 \ee
 Solving \eqref{muj-egi} and comparing \eqref{lambda-pm}, we find that the eigenvalues  of $(-\Delta +I)^{-1}A(\chi)$ are
 \be\label{muj-pm}
 \mu_j^\pm(\chi)=\frac{\lambda^\pm(\chi)}{\sigma_j}, \quad \quad j=0,1,2\cdots.
 \ee
 Recall that $(\lambda^+(O_k^\pm)\cup\lambda^-(O_k^\pm))\cap \Sigma=\emptyset$,  and so  $1$ is not an eigenvalue  of $(-\Delta +I)^{-1}A(\chi)$ for $\chi\in O_k^\pm$. Then the Leray-Schauder index formula gives
\be\label{deg-lk-pm}
\deg(L_k^\pm, B((0,0),r_k^\pm), (0,0))=(-1)^{\gamma_k^\pm},
\ee
where $\gamma_k^\pm$  is the sum of the algebraic multiplicities of the real eigenvalues of $(-\Delta +I)^{-1}A(\chi), \chi\in O_k^\pm$ which are greater than $1$. In the case of $f^\prime(u_0)<0$, since $\lambda^-(\chi)<\sigma_j$ for any $j\geq 1$ and $\chi>\hat{\chi}_1$, we conclude from \eqref{muj-pm} and the properties of $\lambda^+$  that
$$
\gamma_k^+=\sharp(\sigma_0)+\sum_{\sigma_j< \lambda^+(\hat{\chi}_{k+1})} \sharp(\sigma_j), \quad \quad \quad \gamma_k^-=\sharp(\sigma_0)+\sum_{\sigma_j< \lambda^+(\hat{\chi}_k)} \sharp(\sigma_j).
$$
While, in   the case of $f^\prime(u_0)\geq 0$, since $\lambda^-(\chi)<\sigma_j$ for any $j\geq 0$ and $\chi>\hat{\chi}_1$, we conclude from \eqref{muj-pm} and the properties of $\lambda^+$  that
$$
\gamma_k^+=\sum_{\sigma_j< \lambda^+(\hat{\chi}_{k+1})} \sharp(\sigma_j), \quad \quad \quad \gamma_k^-=\sum_{\sigma_j< \lambda^+(\hat{\chi}_k)} \sharp(\sigma_j).
$$
Hence, in either case, we obtain
 \be\label{differ-mult}
 \gamma_k^+-\gamma_k^-=\sharp(\lambda^+(\hat{\chi}_k))=\sharp(\sigma_k).
 \ee
Here the notation $\sharp(\sigma_k)$ denotes the finite algebraic multiplicity of $\sigma_k$.  From \eqref{deg-lk-pm} and \eqref{differ-mult} we deduce that
 \be\label{deg-lk+-lk-}\ba{ll}
&\deg(L_k^+, B((0,0),r_k^+), (0,0))-\deg(L_k^-, B((0,0),r_k^-), (0,0))\\[0.25cm]
&=(-1)^{\gamma_k^-}[(-1)^{\sharp(\sigma_k)}-1].
\ea\ee

Now, if  $\sharp(\sigma_k)$ is an odd number, then  by \eqref{deg-lk-hk} and \eqref{deg-lk+-lk-} the topological structures of $L_k^\pm$ and hence of $H_k^\pm$ change when $\chi$ crosses  $\hat{\chi}_k$. Indeed, by the well-known bifurcation from ``eigenvalues" of  odd multiplicity (cf. \cite{Kra64, CR71, Dancer73, Rab71, Nirenberg74}),  it follows that $\hat{\chi}_k$ is a bifurcation value.  Consequently, there exists a bifurcation branch $\mathcal{C}_k$ containing $(u_0, v_0, \chi_k)$ such that either $\mathcal{C}_k$ is not compact in $X\times X\times \R$ or $\mathcal{C}_k$ contains $(u_0,v_0,\sigma_j)$ with $\sigma_j\neq \sigma_k$.

Case 1: If, for some $k$, the bifurcation branch $\mathcal{C}_k$ is not compact in $X\times X\times \R$,  then $\mathcal{C}_k$ extends to infinity in $\chi$ due to the elliptic regularity that any closed and bounded subset of the solution triple $(u,v,\chi)$ of our chemotaxis system \eqref{para-elli-min-ss} in $X\times X\times \mathbb{R}$  is compact; this can be easily shown by the Sobolev embeddings and results from  \cite [Chapter 3]{La}, see similar discussions in  \cite{CCM12, Xiang15jde2}.   Clearly, in this case, we can find a sequence $\{\chi_k(u_0)\}_{k=1}^\infty$ fulfilling  the statement of the theorem.

Case 2: If, for any $k$, the bifurcation branch $\mathcal{C}_k$ contains $(u_0,v_0,\hat{\chi}_j)$ with $\hat{\chi}_j\neq \hat{\chi}_k$, then we define
$$
\chi_k^-=\inf\{\chi|(u,v,\chi)\in\mathcal{C}_k\}, \quad \quad \quad \chi_k^+=\sup\{\chi|(u,v,\chi)\in\mathcal{C}_k\}<\infty.
$$
Then, for any $\chi\in \cup_{k=0}^\infty(\chi_k^-,\chi_k^+)$, the system \eqref{para-elli-min-ss} has at least one non-constant solution. From this and the fact that $\sigma_k\rightarrow \infty$ and $\hat{\chi}_k=(\lambda^+)^{-1}(\sigma_k)\rightarrow \infty$ as $k\rightarrow \infty$, a sequence $\{\chi_k(u_0)\}_{k=1}^\infty$ satisfying  the description of the theorem can be readily constructed.  Finally, the theorem follows by unifying all $u_0\in Z$.
 \end{proof}
For the constant steady state $(u_0, v_0)$, the length of the existence interval $(\chi_{2k-1},\chi_{2k})$ of solutions is  positive:
$$
|\chi_{2k}-\chi_{2k-1}|\geq \inf\{|(\lambda^+)^{-1}(\sigma_{2k-1})-(\lambda^+)^{-1}(\sigma_j)|: j\neq 2k-1\}>0.
$$
This, joined with $\chi_k\rightarrow \infty$, illustrates  that  the set $P_\chi$ specified in the theorem is unbounded. However, it is yet unknown  whether or not \eqref{para-elli-min-ss} has nonconstant solution  for $\chi$ in the complement of the unbounded set $P_\chi$.

Based on Theorem \ref{bif-thm}, we naturally wish to explore the asymptotic behavior of the solutions $(u,v)$ of \eqref{para-elli-min-ss} when $\chi\rightarrow \infty$. By using the {\it a priori} estimates in Lemma  \ref{pre-est-lem}, we obtain the following result on their asymptotic behavior as $\chi\rightarrow \infty$.
 \begin{theorem} \label{asy-profile} Let $f(u)=au-bu^\theta$ with $a\geq0, b>0, \theta>1$ and $g(u)=\beta u^\kappa$ with
\be\label{theta-over-kappa}
\frac{\theta}{\kappa}>\max\Big\{1, \frac{n}{2}\Big\}, \quad \quad \quad \frac{\theta}{\kappa+1}>\frac{n}{n+1}
\ee and let  $(u_\chi,v_\chi)$ be any positive solution of \eqref{para-elli-min-ss}. Then there is a subsequence $\{\chi_j\}$ with $\lim_{j\rightarrow \infty}\chi_j=\infty$ such that $(u_j,v_j)=(u_{\chi_j},v_{\chi_j})$ fulfills
$$\left\{ \ba{ll}
& \lim_{j\rightarrow \infty} u_j=M,  \quad \quad \mbox{weakly  in  }  L^{\theta}(\Omega),\\[0.2cm]
&\lim_{j\rightarrow \infty} \int_\Omega u_j^\theta=bM|\Omega|/a, \\[0.2cm]
&\lim_{j\rightarrow \infty} v_j=\beta M^\kappa/\alpha,  \quad \mbox{weakly  in  }  W^{2, \frac{\theta}{\kappa}}(\Omega), \\[0.2cm]
&\lim_{j\rightarrow \infty} v_j=\beta M^\kappa/\alpha  \quad \mbox{strongly  in  }  W^{1, p}(\Omega),\\[0.2cm]
&\lim_{j\rightarrow \infty} v_j=\beta M^\kappa/\alpha \quad \mbox{ uniformly  in  } \overline{\Omega}.
 \ea\right. $$
 for some nonnegative constant $M$, where
$$
p<\left\{ \ba{ll}
& \frac{n\theta}{n\kappa-\theta}, \quad  \mbox{ if  } \frac{\theta}{\kappa}<n, \\[0.2cm]
&\infty,  \quad  \quad  \mbox{ if  } \frac{\theta}{\kappa}\geq n.
 \ea\right.
$$
 \end{theorem}
\begin{proof} By Lemma \ref{pre-est-lem}, we see that $\|u_\chi\|_{L^\theta(\Omega)}$ and $\|v_\chi\|_{W^{2, \frac{\theta}{\kappa}}(\Omega)}$ are uniformly bounded with respect to $\chi$. Hence, the reflexivity  and Sobolev embedding allow us to  find a subsequence $\{\chi_j\}$ with $\lim_{j\rightarrow \infty}\chi_j=\infty$ such that $(u_j,v_j)=(u_{\chi_j},v_{\chi_j})$ satisfies
 \be\label{conv-id}\left\{ \ba{ll}
& \lim_{j\rightarrow \infty} u_j=u_\infty,  \quad \mbox{weakly  in  }  L^\theta(\Omega),\\[0.2cm]
&\lim_{j\rightarrow \infty} v_j=v_\infty,  \quad \mbox{weakly  in  } W^{2, \frac{\theta}{\kappa}}(\Omega) \\[0.2cm]
&\lim_{j\rightarrow \infty} v_j=v_\infty,  \quad \mbox{strongly  in  }  W^{1, p}(\Omega),\\[0.2cm]
&\lim_{j\rightarrow \infty} v_k=v_\infty,  \quad \mbox{uniformly  in  }  \overline{\Omega}
 \ea\right. \ee
for some $(u_\infty,v_\infty)\in L^\theta(\Omega)\times W^{2, \frac{\theta}{\kappa}}(\Omega) $ and $v_\infty$ is a (weak) solution of
 \be\label{v-conv}
 - \Delta v_\infty +\alpha v_\infty= \alpha u^\kappa_\infty \text{ in } \Omega, \quad \quad \frac{\partial v_\infty}{\partial \nu}=0 \text{ on  }\partial \Omega.
 \ee
  The last convergence in \eqref{conv-id}  follows from the compact Sobolev embedding $ W^{2, \frac{\theta}{\kappa}}(\Omega) \hookrightarrow C^0(\overline{\Omega})$ since $\theta/\kappa>n/2$.  One can easily infer from \eqref{pre-est} and \eqref{conv-id}  that
 $$
 \|u_\infty\|_{L^\theta(\Omega)}\leq\liminf_{j\rightarrow \infty}  \|u_j\|_{L^\theta(\Omega)}\leq \Bigr(\frac{a}{b}|\Omega|\Bigr)^{\frac{1}{\theta}}.
 $$
 On the other hand, multiplying \eqref{para-elli-min-ss} by $w\in W^{2, \frac{\theta}{\theta-1}}_N(\Omega)$ and dividing by $\chi=\chi_j$ with  $W^{2, \frac{\theta}{\theta-1}}_N(\Omega)=\{u\in W^{2, \frac{\theta}{\theta-1}}(\Omega):\frac{\partial u}{\partial \nu}=0 \text{ on } \partial \Omega\}$, we obtain
  \be\label{test-uj}
  \frac{1}{\chi_j}\int_\Omega u_j\Delta w +\int_\Omega u_j\nabla v_j\nabla w +\frac{1}{\chi_j}\int_\Omega ( au_j-bu_j^\theta)w =0.
  \ee
  Noticing that
  $$\ba{ll}
  &\D\int_\Omega u_j\nabla v_j\nabla w -u_\infty \nabla v_\infty\nabla w\\ [0.25cm]
  &=\D\int_\Omega u_j(\nabla v_j-\nabla v_\infty)\nabla w+\int_\Omega (u_j-u_\infty)\nabla v_\infty\nabla w,
  \ea $$
in the case of $n\kappa>\theta$ and $\theta>n/(n-1)$, we use the  conditions in \eqref{theta-over-kappa} to derive
$$
\ba{ll}
\D\int_\Omega |u_j\nabla w|^{\Bigr(\frac{n\theta}{n\kappa-\theta}\Bigr)^\prime}&=\D\int_\Omega |u_j\nabla w|^{\frac{n\theta}{(n+1)\theta-n\kappa}}\\[0.25cm]
&\leq \Bigr(\D\int_\Omega u_j^\theta\Bigr)^{\frac{n}{(n+1)\theta-n\kappa}}\Bigr(\int_\Omega |\nabla w|^{\frac{n\theta}{(n+1)\theta-n(\kappa+1)}}\Bigr)^{\frac{(n+1)\theta-n(\kappa+1)}{(n+1)\theta-nk}}\\[0.25cm]
&\leq C \Bigr(\D\int_\Omega u_j^\theta\Bigr)^{\frac{n}{(n+1)\theta-n\kappa}}\|w\|_{W^{2,\frac{\theta}{\theta-1}}}^{\frac{n\theta}{(n+1)\theta-nk}}
\ea$$
and that
$$
\int_\Omega |\nabla v_\infty\nabla w|^{(\theta)^\prime}=\int_\Omega |\nabla v_\infty\nabla w|^{\frac{\theta}{\theta-1}}\leq C\|v_\infty\|_{W^{2, \frac{\theta}{\kappa}}}^{\frac{\theta}{\theta-1}}\|w\|_{W^{2,\frac{\theta}{\theta-1}}}^{\frac{\theta}{\theta-1}},
$$
then sending $j\rightarrow \infty$ in \eqref{test-uj}, we conclude from \eqref{conv-id}  that
    \be\label{lim-eq}
 \int_\Omega u_\infty\nabla v_\infty\nabla w=0, \quad \quad \forall w\in W^{2, \frac{\theta}{\theta-1}}_N(\Omega).
  \ee
 Taking the test function $w=v_\infty$ in  \eqref{lim-eq} yields
   \be\label{int-zero}
 \int_\Omega u
_\infty|\nabla v_\infty|^2 =0,
   \ee
   Let $\Omega_\infty=\{x\in \Omega: |\nabla v_\infty(x)|>0\}$. If $|\Omega_\infty|>0$ then $u_\infty=0$ a.e. in $\Omega_\infty$ by \eqref{int-zero}, and so the Fredholm alternative applied to \eqref{v-conv} forces  $v_\infty=0$ in $\Omega_\infty$. This is a contradiction to the definition of $\Omega_\infty$. Whence, $|\Omega_\infty|=0$, i.e., $v_\infty$ is a non-negative constant, say, $\beta M^\kappa/\alpha$,  because  $v_\infty\in   W^{2, \frac{\theta}{\kappa}}(\Omega) \hookrightarrow C(\overline{\Omega})$. Then \eqref{v-conv} again shows  that $u_\infty=M$ is  a nonnegative constant almost everywhere in $\Omega$.  Now, integrating the $u$-equation in \eqref{para-elli-min-ss} entails
    \be\label{uj-theta-lim}
    \int_\Omega u_j^\theta=\frac{a}{b}\int_\Omega u_j\rightarrow \frac{a}{b} M |\Omega|.
    \ee
    This completes the proof the theorem.
    \end{proof}

   \begin{remark}\label{M-unknown}
   Unfortunately,   based on the merely weak convergence  of $\{u_j\}$ in $L^{\theta}(\Omega)$, we are unable to determine the precise values of $M$.  The natural candidates for $M$ are $0$ and $(a/b)^{\frac{1}{\theta-1}}$ because of \eqref{uj-theta-lim}. Indeed, Kuto {\it et al} \cite{KOST12} claimed  either $M=0$ or $a/b$ for the specific choices $\theta=2$ and $\kappa=1$. We here point out that their claim is in general incorrect. The following discussion illustrates the point. Indeed, they claimed  from \eqref{uj-theta-lim}  that $\{u_j\}$ contains a subsequence, still denoted by $\{u_{j}\}$, which converges to $u_\infty$ almost everywhere in $\Omega$ as $j\rightarrow \infty$. However,  the equality \eqref{uj-theta-lim} does not exclude oscillating functions, and hence the claim  is not guaranteed in general. As a simple example,  take
$$
u_j(x)=1+\sin(jx).
$$
Then one can  see that
 \be\label{uj-test-lim}
u_j\rightharpoonup 1 \text{ weakly in } L^2(0,2\pi), \quad \quad \int_0^{2\pi}u_j^2=\frac{3}{2}\int_0^{2\pi}u_j=3\pi.
\ee
This says that $u_j$ satisfies \eqref{uj-theta-lim}  with $a=3, b=2$ and $\theta=2$.  Now, if there is a subsequence $\{j^{'}\}$ of $\{j\}$ such that $u_{j^{'}}\rightarrow 1$ a.e. in $(0,2\pi)$, then the Lebesgue dominated convergence theorem ($0\leq u_j^2\leq 4$) gives
$$
\lim_{j^{'}\rightarrow \infty}\int_0^{2\pi}u_{j^{'}}^2=\int_0^{2\pi}1^2=2\pi,
$$
which is incompatible with \eqref{uj-test-lim}.  Therefore, $u_j$ has no subsequence that converges a.e. to $1$ in $(0,2\pi)$. The other gap of their proof lies in the usage of the Lebesgue convergence theorem without finding the dominating function for $u_{j^{'}} $.  Typically, there is no dominating function for $u_{j^{'}} $,  since, on the one hand, the cells will aggregate when chemotactic effect is strong, and, on the other hand, we would get a stronger convergence  if a dominating function was found. While, a  stronger convergence than that of  Theorem \ref{asy-profile} seems unavailable, since  the regularity, in particular, boundedness  results in Lemma \ref{nonsingular-chi} are not uniform with respect to $\chi$, even in $L^p$-topology. 
\end{remark}

\section{Large time behavior for the K-S model}
For the purpose of easier illustration, we shall first restrict our attention to the large time behavior for a special chemotaxis-growth model with nonlinear production in the chemical equation:
\begin{equation}\label{para-elli-min-lt}\left\{ \begin{array}{lll}
&u_t = \nabla \cdot (\nabla u-\chi u\nabla v)+u(a-bu^\kappa),  &\quad x\in \Omega, t>0, \\[0.2cm]
&0= \Delta v -v+u^\kappa,  &\quad x\in \Omega, t>0, \\[0.2cm]
&\frac{\partial u}{\partial \nu}=\frac{\partial v}{\partial \nu}=0, &\quad x\in \partial \Omega, t>0,\\[0.2cm]
&u(x,0)=u_0(x), &  \quad x\in \Omega,   \end{array}\right.  \end{equation}
where $a\geq 0, b>0, \chi>0$,  $\kappa>0$  and $\Omega\subset \mathbb{R}^n$ is a bounded smooth domain with $n\geq 1$.

For $\kappa =1$,  under the assumption  $\chi<b/2$, Tello and Winkler in \cite{TW07} proved that the solution of \eqref{para-elli-min-lt} converges in $L^\infty$-topology to its constant steady state $(a/b, a/b)$.  Here, for $f(u)=u(a-bu^\kappa)$ with $b>2\chi$,   we extend  the argument  of \cite{TW07} (see also \cite{NT13, STW14})  to show that the solution of \eqref{para-elli-min-lt}  converges in $L^\infty$-topology to its constant steady state $((a/b)^{\frac{1}{\kappa}}, a/b)$ as well. Moreover, we point out that the convergence rate is also exponential. This refined argument can be used to study  a  general chemotaxis-growth model, see Theorem \ref{large time2}.   Our asymptotic  result for \eqref{para-elli-min-lt} as $t$ tends to infinity is as follows.

\begin{theorem}\label{large time} Let  $u_0\in C(\overline{\Omega})$ with $u_0>0$ and let $a, b, \chi, \kappa>0$. If   $b>2\chi$, then the solution of (\ref{para-elli-min-lt}) converges  exponentially in $L^\infty$- topology to the constant equilibrium:
 $$
 \lim_{t\rightarrow \infty} \Bigr(\|u(\cdot,t)-(\frac{a}{b})^{\frac{1}{\kappa}}\|_{L^\infty(\Omega)}+\|v(\cdot,t)-\frac{a}{b}\|_{L^\infty(\Omega)}\Bigr)=0.
 $$ As a result, the equilibrium state  $((\frac{a}{b})^{\frac{1}{\kappa}}, \frac{a}{b})$  is globally asymptotically stable,  and hence the system  (\ref{para-elli-min-lt}) has no nonconstant steady state if the damping effect is strong (i.e., $b>2\chi$).
 \end{theorem}

 In the presence of nonlinear  secretion term $u^\kappa$, there arises several difficulties compared to  the linear case  ($\kappa=1$) \cite{TW07}. Then identification a suitable comparison ODE and demonstration that $u$ will never touch down zero are crucial to overcome those difficulties.
  \begin{proof} The key point is to find an appropriate ODE that is comparable to  (\ref{para-elli-min-lt}):  Let $(\bar{u},\bar{v})=(\bar{u}(t),\bar{v}(t))$ denote the solution of the initial value problem
 \be\label{ode-It}\left\{ \ba{ll}
&\overline{ u}_t=\chi\overline{u}( \overline{u}^\kappa- \underline{u})+ f(\overline{u}) \quad \text{ in } \quad  (0, T_m),  \\[0.25cm]
& \overline{u}(0)=\max\{\max_{x\in\overline{\Omega}}u_0(x), (\frac{a}{b})^{\frac{1}{\kappa}}\},\\[0.25cm]
&(\underline{ u}^{\frac{1}{\kappa}})_t=\chi \underline{ u}^{\frac{1}{\kappa}}( \underline{u}-\overline{u}^\kappa)+ f(\underline{ u}^{\frac{1}{\kappa}})\quad\text{ in }\quad  (0, T_m),  \\[0.25cm]
 &\underline{u}^{\frac{1}{\kappa}}(0)=\min\{\min_{x\in\overline{\Omega}}u_0(x),  (\frac{a}{b})^{\frac{1}{\kappa}}\}
 \ea\right. \ee
defined up to some maximal existence time $T_m\in(0,\infty]$. Hereafter, for later extension purpose, we write $f(u)=u(a-bu^\kappa)$. To ease reading,  the proof is broken into 6 steps.

  Step 1. We show that
  \be\label{u-lu2-com}
0\leq \underline{ u}^{\frac{1}{\kappa}}(t)\leq \overline{u}(t),  \quad\quad \forall t\in [0, T_m).
 \ee
 Clearly,  the problem \eqref{ode-It} has a  unique local solution by the fundamental existence theorem of ODEs.  It follows from $f(0)=0$ that $\underline{u}\equiv 0$ is a sub- solution of the $\underline{u}$- equation in \eqref{ode-It}, the fact $\underline{u}(0)>0$ then readily deduces  $\underline{u}\geq 0$. Indeed, $\underline{u}$ will be shown to be positive  in Step 3 below. 

 If $\overline{u}^\kappa (0)=\underline{u}(0)$ then $(\overline{u}(0),  \underline{u}(0))=(( a/b)^{\frac{1}{\kappa}}, a/b)$,  and so  $(\overline{u}, \underline{u})\equiv (( a/b)^{\frac{1}{\kappa}}, a/b)$ by uniqueness. If $\overline{u}^\kappa(0)>\underline{u}(0)$ and \eqref{u-lu2-com}  were  false,  then there would exist some $t_0>0$ such that $\overline{u}^\kappa(t_0)=\underline{u}(t_0)$. But then, the uniqueness again gives $(\overline{u}(t), \underline{u}(t))\equiv (w(t), w^\kappa(t))$ for all $t\in [0,T_m)$, where $w$ solves $w_t=f(w)$ with  $w(t_0)=\overline{u}^\kappa(t_0)$. This in turn shows $\overline{u}^\kappa(0)=\underline{u}(0)$, a contradiction.

 Step 2.  We show that
  \be\label{u-luzj-com}
 \underline{u}^{\frac{1}{\kappa}}(t)\leq (\frac{a}{b})^{\frac{1}{\kappa}}\leq \overline{u}(t),  \quad\quad \forall t\in [0,  T_{\max}).
 \ee
As a matter of fact, in view of  \eqref{ode-It} and \eqref{u-lu2-com}, one has  $\overline{u}_t\geq f(\overline{u})$. Observe that $\overline{u}(0)\geq (a/b)^{\frac{1}{\kappa}}$ and $f((a/b)^{\frac{1}{\kappa}})=0$,  we obtain $\overline{u}\geq  (a/b)^{\frac{1}{\kappa}}$ by comparison. In the same way, one can readily show  $(\underline{u})^{\frac{1}{\kappa}}\leq  (a/b)^{\frac{1}{\kappa}}$.

  Step 3.  We claim that  $T_{\max}=\infty$ and $\underline{u}(t)>0$ for all $t\geq 0$.

  It suffices to show that $\overline{u}$ is bounded, since $0<\underline{u}\leq a/b$ by Step 2. Indeed, we  conclude from the $\overline{u}$-equation in \eqref{ode-It} and Step 1 that
  $$
  \overline{u}_t=\chi \overline{u}^{\kappa+1}-\chi\overline{u}\underline{u}+f(\overline{u})=\chi \overline{u}^{\kappa+1}-\chi\overline{u}\underline{u}+a\overline{u}-b\overline{u}^{\kappa+1}\leq -(b-\chi)\overline{u}^{\kappa+1}+a\overline{u},
  $$
   which, coupled with the fact $b>\chi$,   implies that $\overline{u}$ is  bounded and 
$$
\overline{u}^\kappa \leq \frac{a}{b-\chi}. 
$$
This combines with  the third equitation in \eqref{ode-It} yield 
$$
\underline{ u}_t=\kappa \underline{ u}\Bigr[\chi( \underline{u}-\overline{u}^\kappa)+  \frac{f(\underline{ u}^{\frac{1}{\kappa}})}{\underline{ u}^{\frac{1}{\kappa}}}\Bigr]\geq\kappa \underline{ u}\Bigr[-(b-\chi)\underline{u}+\frac{(b-2\chi)a}{b-\chi}\Bigr], 
$$
which, together the assumption $b>2\chi, a>0$ and $\underline{u}(0)>0$, directly leads to $\underline{ u}>0$. 

  Step 4. We shall prove  that
\be\label{asy-u-lu}
 |\overline{u}(t)-(\frac{a}{b})^{\frac{1}{\kappa}}|+|\underline{u}(t)-\frac{a}{b}|\rightarrow 0  \quad \text{ exponentially  as } t\rightarrow \infty.
 \ee
Dividing the $\overline{u}$-equation in \eqref{ode-It} by $\overline{u}$ and the $ \underline{u}$-equation by $(\underline{u})^{\frac{1}{\kappa}}$, we  end up with 
$$\left\{ \ba{ll}
&\frac{\overline{ u}_t}{\overline{u}}=\chi( \overline{u}^\kappa- \underline{u})+\frac{f(\overline{u})}{\overline{u}}\quad \text{in  } (0, \infty),  \\[0.25cm]
&\frac{((\underline{ u})^{\frac{1}{\kappa}})_t}{(\underline{ u})^{\frac{1}{\kappa}}}=\chi ( \underline{u}-\overline{u}^\kappa) +\frac{f(\underline{ u}^{\frac{1}{\kappa}})}{\underline{ u}^{\frac{1}{\kappa}}}\quad \text{in  } (0, \infty).
 \ea\right.
 $$
 Subtracting these two equations gives
  \be\label{uo-over-uu}
   \frac{d}{dt}\Bigr(\ln \frac{\overline{u}}{(\underline{u})^{\frac{1}{\kappa }}}\Bigr)=2\chi (\overline{u}^\kappa-\underline{u})+\frac{f(\overline{u})}{\overline{u}}-\frac{f(\underline{ u}^{\frac{1}{\kappa}})}{\underline{ u}^{\frac{1}{\kappa}}}=-(b-2\chi)(\overline{u}^\kappa-\underline{u})\leq 0
  \ee
thanks to the fact $b>2\chi$ and \eqref{u-luzj-com} of Step 2.

Integrating \eqref{uo-over-uu} and using \eqref{u-luzj-com} of Step 2, we obtain
 \be\label{ul-lowerb}
  \underline{u}\geq \frac{\underline{u}(0)}{\overline{u}^\kappa(0)}\overline{u}^\kappa\geq \frac{\underline{u}(0)}{\overline{u}^\kappa(0)}\frac{a}{b}>0.
\ee
  Using mean value theorem, we have 
  $$
  \overline{u}^\kappa- \underline{u}=e^{\ln\overline{u}^\kappa}- e^{\ln \underline{u}}=e^{\ln\xi}(\ln\overline{u}^\kappa- \ln\underline{u})=\xi  \ln \frac{\overline{u}^\kappa}{\underline{u}},
  $$
where $\underline{u}\leq \xi\leq \overline{u}^\kappa$,  we then use \eqref{ul-lowerb}  to sharpen \eqref{uo-over-uu} to
  \be\label{grown-for-uu}
  \frac{d}{dt}\Bigr(\ln \frac{\overline{u}^\kappa}{\underline{u}}\Bigr)\leq =-\kappa \frac{\underline{u}(0)}{\overline{u}^\kappa(0)}\frac{a}{b}(b-2\chi)\ln \frac{\overline{u}^\kappa}{\underline{u}}:=-\epsilon_0\ln \frac{\overline{u}^\kappa}{\underline{u}}.
  \ee
  Solving the differential inequality \eqref{grown-for-uu} gives rise to
  $$
  \ln \frac{\overline{u}^\kappa}{\underline{u}}\leq e^{-\epsilon_0 t} \ln \frac{\overline{u}^\kappa(0)}{\underline{u}(0)}.
  $$
  Passing to the limit as $t\rightarrow \infty$ gives
$$ \kappa\ln \frac{\overline{u}}{(\underline{u})^{\frac{1}{\kappa}}}=
\ln \frac{\overline{u}^\kappa}{\underline{u}}\rightarrow 0  \quad \text{ exponentially  as } t\rightarrow \infty.
  $$
  Therefore,   we deduce from \eqref{u-luzj-com} of Step 2 and \eqref{ul-lowerb}  that 
  $$
 |\overline{u}(t)-(\frac{a}{b})^{\frac{1}{\kappa}}|+|(\underline{u}(t))^{\frac{1}{\kappa}}-(\frac{a}{b})^{\frac{1}{\kappa}}|=\overline{u}(t)-(\underline{u}(t))^{\frac{1}{\kappa}}=(\underline{u}(t))^{\frac{1}{\kappa}}(e^{\frac{1}{\kappa}\ln \frac{\overline{u}^\kappa}{\underline{u}}}-1)\rightarrow 0 
  $$
exponentially as  $ t\rightarrow \infty$,   which  implies  (\ref{asy-u-lu}).

Step 5. The nonlinear production $u^\kappa$ in the $v$-equation of  \eqref{para-elli-min-lt} causes some difficulties. However, we can connect the problem \eqref{ode-It} to our original problem \eqref{para-elli-min-lt} by establishing the key comparison relation
    \be\label{ul-u-uu-key}
 \underline{u}(t)\leq u^\kappa (x,t)\leq \overline{u}^\kappa (t), \quad \quad \quad  \forall (x,t)\in \overline{\Omega}\times [0,\infty).
  \ee
 To this end, let $\overline{U}=u^\kappa-\overline{u}^\kappa$. Then it follows from \eqref{para-elli-min-lt} and  \eqref{ode-It} that
 \be\label{Uu-eq}
\ba{ll}
 &\overline{U}_t-\Delta \overline{U}+\kappa (\kappa-1)u^{\kappa-2}|\nabla u|^2+\chi\nabla \overline{U}\nabla v\\[0.25cm]
&=\kappa\chi (u^\kappa+\overline{u}^\kappa-v)\overline{U}+\kappa\chi (\underline{u}-v)\overline{u}^\kappa+\kappa[u^{\kappa-1}f(u)-\overline{u}^{\kappa-1}f(\overline{u})]\\[0.25cm]
&= \kappa[ (\chi -b)(u^\kappa+\overline{u}^\kappa)- \chi v+a]\overline{U}+\kappa\chi (\underline{u}-v)\overline{u}^\kappa.
\ea
 \ee
 Integrating by parts, one gets
  \be\label{U+-by part}
-\int_\Omega(\overline{U})_+\nabla \overline{U}\nabla v=\frac{1}{2}\int_\Omega (\overline{U})_+^2\Delta v=\frac{1}{2}\int_\Omega (\overline{U})_+^2(v-u^\kappa),\ee
where $u_+=\max\{u, 0\}$ denotes the positive part of $u$.

Now,   we take $(\overline{U})_+$ as a test function in \eqref{Uu-eq} and apply \eqref{U+-by part} to get
 \be\label{diff-Uu}\ba{ll}
 &\D\frac{1}{2}\frac{d}{dt}\int_\Omega (\overline{U})_+^2+\int_\Omega |\nabla (\overline{U})_+|^2+\kappa(\kappa-1)\int_\Omega u^{\kappa-2} |\nabla u|^2 (\overline{U})_+\\[0.25cm]
 &=\D\int_\Omega  \Bigr(\kappa[ (\chi -b)(u^\kappa+\overline{u}^\kappa)- \chi v+a]+\frac{v-u^\kappa}{2}\Bigr) (\overline{U})_+^2\\
 & \ \ \D+\kappa\chi\int_\Omega\overline{u}^\kappa(\underline{u}-v)(\overline{U})_+.
 \ea\ee
 Since $u,v, \overline{u}$ and $\underline{u}$ are bounded, one has
$$
\kappa[ (\chi -b)(u^\kappa+\overline{u}^\kappa)- \chi v+a]+\frac{v-u^\kappa}{2}\leq m_1<\infty
$$
and
$$ \ba{ll}
 &\kappa\chi\D\int_\Omega\overline{u}^\kappa(\underline{u}-v)(\overline{U})_+\leq m_2 \int_\Omega(\underline{u}-v)_+(\overline{U})_+\\
 &\ \ \leq \D\frac{m_2}{2}\int_\Omega(\underline{u}-v)_+^2+\frac{m_2}{2}\int_\Omega(\overline{U})_+^2. 
\ea$$
 We then deduce from \eqref{diff-Uu} that
\be \ba{ll}\label{diff-Uuf}
 &\D\frac{1}{2}\frac{d}{dt}\int_\Omega (\overline{U})_+^2+\int_\Omega |\nabla (\overline{U})_+|^2+\kappa(\kappa-1)\int_\Omega u^{\kappa-2} |\nabla u|^2 (\overline{U})_+\\
 &\ \ \leq m_3\D\int_\Omega (\overline{U})_+^2+m_4\int_\Omega(\underline{u}-v)_+^2,
\ea\ee
 where and hereafter $m_i$ denotes a generic constant independent of $t$.

In the case of $\kappa\in (0, 1)$, we need to control  the third integral on the left-hand side of \eqref{diff-Uuf}. To this aim, we shall show that the  $u$-component  of \eqref{para-elli-min-lt} in fact  has a positive lower bound over $\bar{\Omega}\times [t_0,\infty)$ for any $t_0>0$:
\be\label{u-lower b}
\inf_{\Omega\times(t_0,\infty)}u:=m_5>0.
\ee
Its proof is little involved  than that of $\underline{u}>0$ in Step 3.  Indeed, expanding out the $u$-equation and using the $v$-equation of \eqref {para-elli-min-lt},  we get
$$
u_t=\Delta u-\chi\nabla u\nabla v-\chi u(v-u^\kappa)+f(u)\leq \Delta u-\chi\nabla u\nabla v-(b-\chi)u^{\kappa+1}+au.
$$
By comparing with the corresponding ODE or maximum principle, we see from $b>\chi$  that
$$
\limsup_{t\rightarrow \infty} \max_{x\in\bar{\Omega}}u(x,t)\leq \Bigr(\frac{a}{b-\chi}\Bigr)^{\frac{1}{\kappa}}.
$$
Then the maximum principle and Hopf lemma applied to $-\Delta v+v=u^\kappa$  yield
$$
\limsup_{t\rightarrow \infty} \max_{x\in\bar{\Omega}}v\leq \frac{a}{b-\chi}.
$$
This in turn gives,  for sufficiently large time $t$, that
$$
u_t=\Delta u-\chi\nabla u\nabla v-\chi u(v-u^\kappa)+f(u)\geq \Delta u-\chi\nabla u\nabla v-(b-\chi)u^{\kappa+1}+\frac{(b-2\chi)a}{b-\chi} u.
$$
Then the  fact that $b>2\chi$ and $a>0$ coupled with maximum principle and Hopf lemma again show
$$
\liminf_{t\rightarrow \infty} \min_{x\in\bar{\Omega}}u(x,t)\geq \Bigr[\frac{(b-2\chi)a}{(b-\chi)^2}\Bigr]^{\frac{1}{\kappa}}>0.
$$
On the other hand, the strong maximum principle and Hopf lemma allow  one to conclude  $u>0$ on $\bar{\Omega}\times (0,\infty)$. Thus,  $u$ has a positive lower bound over $\bar{\Omega}\times [t_0,\infty)$ for any $t_0>0$ ($t_0$ can be taken to be zero if $\inf_{\Omega}u_0>0$), leading to \eqref{u-lower b}.

In virtue of  \eqref{u-lower b}, we estimate
\be \ba{lll}\label{difficult-term}
\D\int_\Omega u^{\kappa-2} |\nabla u|^2 (\overline{U})_+&=\D\int_\Omega \frac{1}{\kappa^2u^\kappa}|\nabla (\overline{U})_+|^2(\overline{U})_+\\[0.25cm]
&\D\leq \frac{1}{\kappa^2m_5^\kappa} \int_\Omega |\nabla (\overline{U})_+|^2(\overline{U})_+\\[0.3cm]
&=-\D\frac{1}{2\kappa^2m_5^\kappa} \int_\Omega (\overline{U})_+^2\Delta (\overline{U})_+.
\ea
\ee
Since $u$ and $v$ are bounded, using $H^2$ estimate to the $v$-equation and parabolic boundary $L^p$-estimates and Schauder estimates (cf. \cite{La}) to the $u$-equation in \eqref{para-elli-min-lt}, we see that all spatial partial derivatives of $u$ and $v$ up to order two are bounded in $\overline{\Omega}\times [0,\infty)$.  Consequently,  for $\kappa>0$, no matter $\kappa\geq 1$ or not,  we infer from \eqref{diff-Uuf} and \eqref{difficult-term} that
  \be\label{diff-Uuf}
\frac{1}{2}\frac{d}{dt}\int_\Omega (\overline{U})_+^2+\int_\Omega |\nabla (\overline{U})_+|^2
 \leq m_6\int_\Omega (\overline{U})_+^2+m_7\int_\Omega(\underline{u}-v)_+^2.
\ee
Carrying out the similar procedure as above to  $\underline{U}=u^k-\underline{u}$ and keeping in mind \eqref{ode-It} and \eqref{u-lower b},  we obtain
\be\label{diff-Ulf}
 \frac{1}{2}\frac{d}{dt}\int_\Omega (\underline{U})_-^2+\int_\Omega |\nabla (\underline{U})_-|^2
 \leq m_8\int_\Omega (\underline{U})_-^2+m_9\int_\Omega(v-\overline{u}^\kappa)_+^2, \ee
where $u_-=\min\{u, 0\}$ denotes  the negative part of $u$.

The $v$-equation in \eqref{para-elli-min-lt} implies $-\Delta v+(v-\overline{u}^\kappa)=\overline{U}$. Hence, by choosing $(v-\overline{u}^\kappa)_+$ as a test function and using Cauchy-Schwarz  inequality, we obtain
$$
\int_\Omega |\nabla (v-\overline{u}^\kappa)_+|^2 +\int_\Omega  (v-\overline{u}^\kappa)_+^2\leq\int_\Omega  (\overline{U})_+(v-\overline{u}^\kappa)_+ \leq \frac{1}{2}\int_\Omega  (\overline{U})_+^2+\frac{1}{2}\int_\Omega(v-\overline{u}^\kappa)_+^2,
$$
which directly gives
 \be\label{Uu-lowerb}
  \int_\Omega(v-\overline{u}^\kappa)_+^2\leq \int_\Omega  (\overline{U})_+^2.
 \ee
 Similarly, taking $(v-\underline{u})_-$ as a test function in $-\Delta v+(v-\underline{u})=\underline{U}$, we obtain
  \be\label{Ul-lowerb}
   \int_\Omega(\underline{u}-v)_+^2=\int_\Omega(v-\underline{u})_-^2\leq \int_\Omega  (\underline{U})_-^2.
 \ee
 We finally infer from \eqref{diff-Uuf}-\eqref{Ul-lowerb}  that
$$
\frac{d}{dt}\bigg(\int_\Omega (\overline{U})_+^2+\int_\Omega(\underline{U})_-^2
 \bigg)\leq m_{10}\bigg(\int_\Omega (\overline{U})_+^2+\int_\Omega(\underline{U})_-^2\bigg).
 $$
Observing that $(\overline{U}(0))_+=0=(\underline{U}(0))_-$, we solve  from the above inequality that
$$
(\overline{U})_+=0=(\underline{U})_-,
$$
which establishes  the key inequality \eqref{ul-u-uu-key}.

Step 6. To complete the proof of the theorem, we conclude from \eqref{ul-u-uu-key}   and \eqref{asy-u-lu}, exponentially,   that
 \be\label{u-lim}
\lim_{t\rightarrow \infty}\|u(\cdot,t)-(\frac{a}{b})^{\frac{1}{\kappa}}\|_{L^\infty(\Omega)}=0.
 \ee
The maximum principle applied to the equation $-\Delta v+v=u^\kappa$ shows
$$
\inf_{y\in\bar{\Omega}}u^\kappa(y,t)\leq v(x,t)\leq \sup_{y\in\bar{\Omega}}u^\kappa(y,t), \quad \forall (x,t)\in \overline{\Omega}\times [0,\infty).
$$
Sending $t\rightarrow \infty$ and using \eqref{u-lim}, we arrive at, exponentially, 
$$
\lim_{t\rightarrow \infty}\|v(\cdot,t)-\frac{a}{b}\|_{L^\infty(\Omega)}=0,
$$
which completes the proof.
\end{proof}
 For  a general chemotaxis-growth  model  of \eqref{para-elli-min-lt} of the form
\begin{equation}\label{para-elli-min-general}\left\{ \begin{array}{lll}
&u_t = \nabla \cdot (\bigtriangledown u-\chi u\nabla v)+f(u),  &\quad x\in \Omega, t>0, \\[0.2cm]
&0= \bigtriangleup v -v+u^\kappa,  &\quad x\in \Omega, t>0, \\[0.2cm]
&\frac{\partial u}{\partial \nu}=\frac{\partial v}{\partial \nu}=0, &\quad x\in \partial \Omega, t>0,\\[0.2cm]
&u(x,0)=u_0(x), & \quad x\in \Omega,    \end{array}\right.  \end{equation}
where the growth source $f$ satisfies
\be\label{f-growth-kappa}
f(u)\leq \sigma-\mu u^{\kappa+1}, \forall u\geq 0\text { with } \sigma \geq 0, \mu>0  \text{ and  } f(0)\geq0.
\ee
We wish to show, for suitably small $\chi$, the large time behavior of \eqref{para-elli-min-general} is comparable  to the corresponding  ODE+algebraic system:
\begin{equation}\label{ode-lt}\left\{ \begin{array}{lll}
&p_t = f(p),  &\quad  t>0, \\[0.2cm]
&0=  -q+p^\kappa, &\quad   t>0, \\[0.2cm]
&p(0)=p_0>0.   \end{array}\right.  \end{equation}
To that end, let us arrange the finitely many zero points of $f$ in an increasing order as  $Z=\{z\geq 0: f(z)=0\}=\{ z_0=0, z_1,z_2,\cdots,z_k\}$. Then simple phase line analysis yields that $(z_{2j+1}, z_{2j+1}^\kappa)$  is locally asymptotically stable and  $(z_{2j},z_{2j}^\kappa) $ is locally unstable stable for the associated ODE+algebraic system, respectively.   To carry over such stability to the chemotaxis-growth  model  \eqref{para-elli-min-general}, besides \eqref{f-growth-kappa}, we need a stronger condition on $f$: there exists $\eta_0>0$ such that,  for any $r>0,s>0$ with $r\leq z_j\leq  s$, one has
 \be\label{f-diff-strong}
\frac{f(s)}{s}-\frac{f(r)}{r}\leq -(2\chi+\eta_0) (s^\kappa-r^\kappa).
\ee
In the case of $\kappa\neq 1$ (nonlinear production), we need  impose a condition  to preclude that $u$ touches  down zero at infinity, which will be the case if the infimum of the solution  to
\be\label{u-low-b-ODE} \left\{ \begin{array}{ll}
&y_t = f(y)+\chi y^{\kappa+1}-\chi z_\infty^\kappa y, \\[0.2cm]
&y(0)=\min_{x\in\bar{\Omega}} u_0(x)   \end{array}\right.  \ee
converges to a positive constant as $t\rightarrow \infty$, which may be implied by \eqref{f-diff-strong}. Here, $0<z_\infty=\lim_{t\rightarrow \infty} z(t)$ and $z(t)$ is the  solution of
$$
\left\{ \begin{array}{ll}
&z_t = f(z)+\chi z^{\kappa+1}, \\[0.2cm]
&z(0)=\max_{x\in\bar{\Omega}} u_0(x).   \end{array}\right.
$$

With minor  modifications of the proof of Theorem \ref{large time}, one can derive the following general convergence result for \eqref{para-elli-min-general}.
 \begin{theorem}\label{large time2} Let $f$ fulfill \eqref{f-growth-kappa} with $\mu>\chi$ and \eqref{f-diff-strong} and let the solution $y(t)$ of \eqref{u-low-b-ODE} satisfy  $\liminf_{t\rightarrow \infty} y(t)=y_\infty>0$. Then,  for any locally stable equilibrium $(z_j,z_j^2)$ of the ODE+algebraic system \eqref{ode-lt} with  $1\leq j\leq k$ and for any positive  initial data $u_0\in C^0(\overline{\Omega})$ with
\be\label{initial-u0}
z_{j-1}< \min_{\bar{\Omega}}u_0\leq \max_{\bar{\Omega}}u_0<z_{j+1},
\ee
 the solution of (\ref{para-elli-min-general})  converges exponentially  in $L^\infty$- topology to the constant equilibrium  $(z_j,z_j^2)$ as $\rightarrow \infty$:
 $$
 \lim_{t\rightarrow \infty} \Bigr(\|u(\cdot,t)-z_j\|_{L^\infty(\Omega)}+\|v(\cdot,t)-z_j^\kappa\|_{L^\infty(\Omega)}\Bigr)=0.
 $$
Here, we augmented the notation $z_0=0$ and $z_{k+1}=\infty$.  As a result, the constant steady state $(z_j,z_j^2)$  is locally exponentially asymptotically stable.
 \end{theorem}

\textbf{Acknowledgments}    We thank Prof. Michael Winkler for personal communications, useful discussions and suggestions on the manuscript. Z. Wang  was supported by the Hong Kong RGC ECS (early career scheme) grant No. 509113 and an internal grant  G-YBCS from the Hong Kong Polytechnic University. T. Xiang  was  funded by China Postdoctoral Science Foundation (No.2015M570190), the Fundamental Research Funds for the Central Universities, and the Research Funds of Renmin University of China (No.15XNLF10).



\begin{thebibliography}{99}


\bibitem{Al}  N. Alikakos, $L^p$ bounds of solutions of reaction-diffusion equations,  Comm. Partial Differential Equations 4 (1979),  827--868.


\bibitem{BH13} K. Baghaei and M.  Hesaaraki, Global existence and boundedness of classical solutions for a chemotaxis model with logistic source, C. R. Math. Acad. Sci. Paris 351 (2013), 585--591.


\bibitem{BBTW15} N. Bellomo, A. Bellouquid, Y. Tao, M. Winkler, Toward a mathematical theory of Keller-Segel models of pattern formation in biological tissues,  Math. Models Methods Appl. Sci. 25 (2015),  1663--1763.



\bibitem{CCM12} R. Cantrell, C.  Cosner and R.  Man\'{a}sevich, Global bifurcation of solutions for crime modeling equations, SIAM J. Math. Anal.  44  (2012),  1340-1358.

 \bibitem{CZ13} X. Cao and S. Zheng, Boundedness of solutions to a quasilinear parabolic-elliptic Keller-Segel system with logistic source,  Math. Methods Appl. Sci.  37  (2014),   2326--2330.

 \bibitem{Cao14} X. Cao, Boundedness in a quasilinear parabolic–parabolic Keller-Segel system with logistic source, J. Math. Anal. Appl. 412 (2014), 181--188.


\bibitem{CR71} M. G. Crandall and P. H. Rabinowitz, Bifurcation from simple eigenvalues, J. Functional Analysis, 8 (1971), 321-340.

\bibitem{Dancer73} E.N.  Dancer, On the structure of solutions of non-linear eigenvalue problems, Indiana Univ. Math. J. 23 (1973/74), 1069-1076.


 \bibitem{Fried} A. Friedman,  Partial differential equations. Holt, Rinehart and Winston, New York-Montreal, Que.-London, 1969.

\bibitem{Herrero-Vela97b} M.A. Herrero and J.J.L. Vel\^{a}zquez, Finite-time  aggregation into a single point in a reaction-diffusion system, Nonlinearity, 10 (1997), 1739-1754.

  \bibitem{Hi} T. Hillen and K. Painter, A user's  guide to PDE models for chemotaxis,  J. Math. Biol., 58 (2009), 183--217.


  \bibitem{Ho1} D. Horstmann, From 1970 until now: the Keller-Segal model in chaemotaxis and its consequence I,  Jahresber   DMV, 105 (2003), 103--165.


  \bibitem{HW06} D. Horstmann and M.  Winkler,  Boundedness vs. blow-up in a chemotaxis system, J. Differential Equations 215 (2005),  52--107.

  \bibitem{Jager} W. J\"{a}ger and S. Luckhaus, On explosions of solutions to a system of partial differential equations modeling chemotaxis, Trans. Amer. Math. Soc. 329 (1992), 819-824.

   \bibitem{JXpre} H. Jin and T. Xiang, A blow-up criterion for fully parabolic Keller-Segel chemotaxis systems, in preprint.

\bibitem{Kra64} M.A. Krasnosel'skii, Topological methods in the theory of nonlinear integral equations, New York 1964.

\bibitem{Ke} E. Keller and L. Segel, Initiation of slime mold aggregation viewed as an instability, J. Theoret Biol.,  26 (1970), 399--415.




\bibitem{KS08} R. Kowalczyk, Z.  Szymanska, On the global existence of solutions to an aggregation model, J. Math. Anal. Appl.  343  (2008),  no. 1, 379--398.

\bibitem{KOST12} K. Kuto, K.  Osaki,  T. Sakurai and T.  Tsujikawa,
Spatial pattern formation in a chemotaxis-diffusion-growth model,  Phys. D  241  (2012),   1629--1639.

\bibitem{La}  O. Ladyzhenskaya, V. Solonnikov and N.  Uralceva,  Linear and Quasilinear Equations of Parabolic Type, AMS, Providence, RI, 1968.


\bibitem{La15-JDE}J. Lankeit, Eventual smoothness and asymptotics in a three-dimensional chemotaxis system with logistic source,  J. Differential Equations  258  (2015), 1158--1191.

\bibitem{LN99} Y. Lou and W. M.  Ni,  Diffusion vs cross-diffusion: an elliptic approach, J. Differential Equations  154  (1999), 157--190.


 \bibitem{Mu02} J.D. Murray,  Mathematical biology. I. An introduction. Third edition. Interdisciplinary Applied Mathematics, 17. Springer-Verlag, New York, 2002. xxiv+551 pp.

 \bibitem{Na95} T. Nagai,  Blow-up of radially symmetric solutions to a chemotaxis system,  Adv. Math. Sci. Appl. 5 (1995), 581--601.


 \bibitem{EN08} E. Nakaguchi and M.  Efendiev, On a new dimension estimate of the global attractor for chemotaxis-growth systems, Osaka J. Math. 45 (2008), 273--281.

\bibitem{NO11} E. Nakaguchi and  Osaki,  Global existence of solutions to a parabolic-parabolic system for chemotaxis with weak degradation,  Nonlinear Anal. (2011), 286--297.

\bibitem{NO13} E. Nakaguchi and  Osaki,  Global solutions and exponential attractors of a parabolic-parabolic system for chemotaxis with subquadratic degradation, Discrete Contin. Dyn. Syst. Ser. B 18  (2013), 2627--2646.


\bibitem{NT13} M. Negreanu and J. Tello,
On a comparison method to reaction-diffusion systems and its applications to chemotaxis,  Discrete Contin. Dyn. Syst. Ser. B  18  (2013),  2669--2688.


\bibitem{Nirenberg66} L. Nirenberg,   An extended interpolation inequality, 
Ann. Scuola Norm. Sup. Pisa (3) 20 (1966),  733-737.

\bibitem{Nirenberg74} L.  Nirenberg, Topics in nonlinear functional analysis,  Courant Institute of Mathematical Sciences,  New York University, New York, 1974.

\bibitem{OTYM02} K. Osaki,T.  Tsujikawa, A. Yagi, M.  Mimura, Exponential attractor for a chemotaxis-growth system of equations,  Nonlinear Anal. 51 (2002), 119-144.

\bibitem{Rab71} P. Rabinowitz,  Some global results for nonlinear eigenvalue problems,  J. Functional Analysis 7 (1971), 487--513.


\bibitem{Smoller83} I.  Smoller,   Shock waves and reaction-diffusion equations,  Springer-Verlag, New York-Berlin, 1994.

\bibitem{STW14} C. Stinner, J. Tello and M. Winkler,  Competitive exclusion in a two-species chemotaxis model, J. Math. Biol.  68  (2014),  1607--1626.

 \bibitem{TW12} Y. Tao and M. Winkler, Boundedness in a quasilinear parabolic-parabolic Keller-Segel system with subcritical sensitivity,  J. Differential Equations 252 (2012),  692--715.

 \bibitem{TW14-JDE} Y. Tao and M. Winkler,  Persistence of mass in a chemotaxis system with logistic source, J. Differential Equations  259  (2015),  6142--6161.

\bibitem{TW15-ZAMP} Y.  Tao and M. Winkler,  Boundedness and decay enforced by quadratic degradation in a three-dimensioanl chemotaxis-fluid system,  Z. Angew. Math. Phys.  66  (2015),  2555--2573. 

 \bibitem{TW07} J. Tello and M.  Winkler,  A chemotaxis system with logistic source, Comm. Partial Differential Equations 32 (2007),  849--877.



 \bibitem{WMZ14} L.  Wang, C.  Mu and P.  Zheng, On a quasilinear parabolic-elliptic chemotaxis system with logistic source, J. Differential Equations 2556 (2014),  1847--1872.


 \bibitem{Wang-review} Z. Wang, Mathematics of traveling waves in chemotaxis. Discrete Contin. Dyn. Syst-Series B., 18  (2013), 601-641.



 \bibitem{WWW12} Z. Wang, M. Winkler, D. Wrzosek, Global regularity vs. infinite-time singularity formation in a chemotaxis model with volume-filling effect and degenerate diffusion, SIAM J. Math. Anal. 44 (2012), 3502--3525.

 \bibitem{Win08} M. Winkler, Chemotaxis with logistic source: very weak global solutions and their boundedness properties, J. Math. Anal. Appl. 38  (2008), 708-729.

\bibitem{Win100} M. Winkler,  Aggregation vs. global diffusive behavior in the higher-dimensional Keller-Segel model, J. Differential Equations, 248  (2010), 2889-2905.

\bibitem{Win10}M. Winkler,  Boundedness in the higher-dimensional parabolic-parabolic chemotaxis system with logistic source,  Comm. Partial Differential Equations 35 (2010),  1516-1537.




\bibitem{WD10} M. Winkler and K. Djie, Boundedness and finite-time collapse in a chemotaxis system with volume-filling effect, Nonlinear Anal. 72  (2010),  1044--1064.

\bibitem{Win11} M. Winkler, Blow-up in a higher-dimensional chemotaxis system despite logistic growth restriction, J. Math. Anal. Appl. 384  (2011), 261--272.

\bibitem{Win13}M. Winkler, Finite-time blow-up in the higher-dimensional parabolic-parabolic Keller-Segel system, J. Math. Pures Appl. 100  (2013), 748--767.

 \bibitem{Win15-jde} M. Winkler,  Global asymptotic stability of constant equilibria in a fully parabolic chemotaxis system with strong logistic dampening, J. Differential Equations  257  (2014), 1056--1077.

\bibitem{Xiang15}T. Xiang, Boundedness and global existence in the higher-dimensional parabolic–parabolic chemotaxis system with/without growth source,  J. Differential Equations  258  (2015),  4275--4323.

\bibitem{Xiang15jde2} T. Xiang, On a  class of  Keller-Segel chemotaxis systems with cross-diffusion, J. Differential Equations 259  (2015),   4273--4326.

\bibitem{Xiang15-pre}T. Xiang,  How strong  a logistic damping can  prevent blow-up for the minimal  Keller-Segel chemotaxis systems? Preprint.

\bibitem{ZL15-ZAMP} Q. Zhang and Y. Li, Boundedness in a quasilinear fully parabolic Keller-Segel system with logistic source, Z. Angew. Math. Phys.  66  (2015),  2473--2484.

  \bibitem{Zh15} J. Zheng,  Boundedness of solutions to a quasilinear parabolic–elliptic Keller-Segel system with logistic source., J. Differential Equations  259  (2015),   120--140.



\end{thebibliography}
\end{document}